\documentclass[a4paper,10pt]{article}

\usepackage{amssymb}
\usepackage{fancyhdr}		
\newcommand{\fin}{\hfill$\square$}

\usepackage[T1]{fontenc}
\usepackage[latin1]{inputenc}
\usepackage[a4paper,textwidth=14cm,textheight=22cm]{geometry}
\usepackage{url}	
\usepackage{graphicx}	
\usepackage{amssymb}
\usepackage{amsmath}
\usepackage{amsthm}
\usepackage[T1]{fontenc}
\usepackage{ae}

\newtheorem{theo}{Theorem}[section]
\newtheorem{cor}[theo]{Corollary }
\newtheorem{lemm}[theo]{Lemma}

\newtheorem{defi}{Definition}[section]
\newtheorem{bei}{Example}[section]
\newtheorem{prop}[theo]{Proposition}

\newtheorem{rmq}[theo]{Remark}

\newtheorem{conj}{Conjecture}

\title{Asymptotic behavior of Cauchy hypersurfaces in constant curvature space-times}
\author{Mehdi Belraouti}
\begin{document}
\maketitle

\noindent{\bf Abstract.}
We study the asymptotic behavior of convex Cauchy hypersurfaces on maximal globally hyperbolic spatially compact space-times of constant curvature. We generalise the result of \cite{mehdi1} to the (2+1) de Sitter and anti de Sitter cases. We prove that in these cases the level sets of quasi-concave times converge in the Gromov equivariant topology, when time goes to $0$, to a real tree. Moreover, this limit does not depend on the choice of the time function. We also consider the problem of asymptotic behavior in the flat $(n+1)$ dimensional case. We prove that the level sets of quasi-concave times converge in the Gromov equivariant topology, when time goes to $0$, to a $CAT(0)$ metric space. Moreover, this limit  does not depend on the choice of the time function.

\tableofcontents

\section{Introduction}
Space-times of constant curvature occupy an important place in Lorentzian geometry.  Despite their trivial local geometry, these spaces have a very rich global geometry  and constitute an important family of space-times in which we  hope to understand many fundamental questions. The existence of time functions with levels of prescribed geometry constitutes one of these questions both from the geometrical and the physical point of view. We refer to these functions as geometric time functions. This question was amply studied in the literature in the works of Andersson, Barbot, B\'eguin, Benedetti, Bonsante, Fillastre, Galloway, Guadignini, Howard, Moncrief, Seppi, Zeghib (we cite for example \cite{anderson1}, \cite{anderson6}, \cite{anderson2}, \cite{BBZ1}, \cite{bonsante1}, \cite{BBZ2}, \cite{barbot1}, \cite{beneguadinini}, \cite{Fillastre1}, \cite{Seppi}). The main object of this article is to study the asymptotic behavior of geometric time functions levels. 

Recall that a Lorentzian manifold is a differentiable manifold endowed with a pseudo-Riemannian metric of signature $(-,+,...,+)$. A space-time is an oriented and chronologically oriented Lorentzian manifold. A space-time is said to be globally hyperbolic ($GH$) if it possesses a function, called Cauchy time function, which is strictly increasing along causal curves (curves for which the norms of the tangent vectors are non positive) and surjective on inextensible causal curves. The levels of such function are called Cauchy hypersurfaces. If in addition the Cauchy time function is proper then we say that the space-time is globally hyperbolic spatially compact and we write $GHC$. By a classical result of Geroch \cite{Geroch30}, every $GH$ space-time is diffeomorphic to the product of a Cauchy  hypersurface $S$ by an interval $I$ of $\mathbb{R}$. 
A globally hyperbolic spatially compact space-time, solution of the Einstein equation, is said to be maximal if it doesn't extend to a constant curvature $GHC$ space-time which is also solution of the Einstein equation. A maximal globally hyperbolic spatially compact space-time is denoted by $MGHC$. A space-time is said to be of constant curvature if it is endowed with a $(G,X)$ structure where $X$ is a constant model space and $G$ his isometry group. Recall that the models of constant curvature space-times are the Minkowski space $\mathbb{R}^{1,n}$, the anti de Sitter space $AdS_{n}$ and the de Sitter space $dS_{n}$.\\

In \cite{mess1}, Mess gives a full classification of $MGHC$ space-times in the $2+1$ flat and anti de Sitter cases giving rise in the same time  to a particular interest for $MGHC$ space-times of constant curvature. Following Mess work's  Scannell, Barbot, B\'eguin, Bonsante and Zeghib (\cite{scannell1}, \cite{barbot1},\cite{bonsante1}, \cite{BBZ1}) complete this classification in all constant curvature and all dimension cases. In the $2+1$ special case Mess \cite{mess1}, Benedetti and Bonsante \cite{benebonsan1} proved that there is a one to one correspondence between measured geodesic laminations on a given closed hyperbolic surface $S$ and $MGHC$ constant curvature space-times admitting a Cauchy surface diffeomorphic to $S$. 
\\

The $MGHC$ space-times of constant curvature have the particularity to possess remarkable geometric time functions:\\ 
1) The cosmological time, which is defined at a point $p$ as the supremum of length of past causal curves starting at $p$. It gives  a simple and important first example of quasi-concave time functions i.e those which the levels are convex, to which all other time functions can be compared (see \cite{bonsante1}, \cite{scannell1}).\\
2) The $CMC$ time function i.e a time function where the levels have constant mean curvature. The existence and uniqueness of such function in a given space-time was studied by Andersson, Barbot, Béguin and Zeghib in the flat, de Sitter and anti de Sitter cases \cite{bbz4}, \cite{bbz5}, \cite{BBZ1}. These functions  define  a regular foliation and play an important role in physics. In the flat case they have the particularity to be quasi-concave.\\
3) The $k$-time (dimension $2+1$) i.e a time function where the levels have constant Gauss curvature.  The existence and uniqueness of such function in a given space-time was done by Barbot, Béguin and Zeghib \cite{BBZ2}. They are by definition quasi-concaves.

Up to inversion of time orientation, these space-times have also  the particularity to be geodesically complete in the future (or in the past), but on the other hand often incomplete in the past (or in the future); we say that they admit an initial singularity. Giving a mathematical sense to this notion constitutes an important problem in general relativity (see \cite{Penrose12}, \cite{Hanwking14}, \cite{Hawking15}, \cite{Penrose13}, \cite{hawking}, \cite{oneill}). There are in the literature different ways to attach a boundary to a  space-time; we cite for example the Penrose boundary \cite{Penrose1}, the b-boundary \cite{Schmidt}. However, these constructions are not unique in general and all have disadvantages. We hope, through the study of  asymptotic behavior of Cauchy hypersurfaces, to give a more intrinsic meaning to this notion of initial singularity.

Let $M$ be a $MGHC$ space-time of constant curvature. A Cauchy time function $T:M\rightarrow \mathbb{R}$ defines naturally a $1$-parameter family  $(T^{-1}(a),g_{a})_{a\in \mathbb{R}}$ of Riemannian manifolds  or equivalently  a $1$-parameter family $(T^{-1}(a),d_{a})_{a\in \mathbb{R}}$ of metric spaces. One can ask the natural important question of asymptotic behavior of this family  with respect to the time in the following two cases: when time goes to $0$ and when it goes towards infinity. In our case we consider the equivalent equivariant problem: the asymptotic behavior of the $\pi_{1}(M)$-equivariant family $(\pi_{1}(M),\tilde{T}^{-1}(a),\tilde{d}_{a})_{a\in \mathbb{R}}$. Several notions of topology appear when we deal with the convergence of equivariant metric spaces. In this article our favorite convergences will be the compact open convergence and the Gromov equivariant convergence \cite{paulin1}, \cite{paulin2}.   

The study of such problem was first initiated by Benedetti-Guadagnini \cite{beneguadinini}. They noticed that the cosmological levels of $MGHC$ flat space-times converge, when time goes to $0$, to the real tree dual to the measured geodesic lamination associated to $M$. This problem was finally treated by Bonsante, Benedetti in \cite{bonsante1}, \cite{benebonsan1}. In the case of the $CMC$ time Benedetti-Guadagnini \cite{beneguadinini} conjectured that in a flat globally hyperbolic spatially compact non elementary maximal space-time $M$ of dimension $2+1$, the level sets of the $CMC$ time converge when time goes to $0$ to the real tree dual to the measured geodesic lamination associated to $M$ and when time goes to the infinity to the hyperbolic structure associated to $M$. In \cite{anderson3} Andersson gives a positive answer to the Benedetti-Guadagnini conjecture in the case of simplicial flat space-time. A complete positive answer to this conjecture is given in \cite{mehdi1}.

In the $2+1$ case, one can formulate the asymptotic problem in the Teichmüller space. Let $S$ be a closed hyperbolic surface and $M$ be a constant curvature $MGHC$ space-time admitting a Cauchy surface diffeomorphic to $S$. A Cauchy time function $T:M\rightarrow ]0,+\infty[$ defines naturally
a curve $(S,g^{T}_{a})_{a}$ in the space $\operatorname{Met}(S)$ of Riemannian metrics of $S$. This allows us to study the
behavior of the projection  curve $(S,[g^{T}_{a}])_{a}$  of $(S,g^{T}_{a})_{a}$ in the Teichmüller space $\operatorname{Teich}(S)$ which is, as a
topological space, much more pleasant to study than $\operatorname{Met}(S)$. In the flat case and  thanks to the work of Benedetti and Bonsante \cite{benebonsan1}, one can identify the curve $(S,[g^{T_{cos}}_{a}])_{a}$ in $\operatorname{Teich}(S)$ associated to the cosmological time $T_{cos}$. It corresponds to the grafting curve $(\operatorname{gra}_{\frac{\lambda}{a}}(S))_{a}$ defined by the measured geodesic lamination $(\lambda,\mu)$ associated to $M$. The curve $(\operatorname{gra}_{\frac{\lambda}{a}}(S))_{a}$ is real analytic and converges when
time goes to $0$ to the point, in the Thurston
boundary of the Teichmüller space $\operatorname{Teich}(S)$, corresponding to the measured geodesic lamination $(\lambda,\mu)$, and when time goes to $+\infty$, to the hyperbolic
structure $\mathbb{H}^{2}/\pi_{1}(M)$.

In the case of the $CMC$ time $T_{cmc}$, Moncrief \cite{Moncrief2} proved that the curve $(S,[g^{T}_{a}])_{a}$ is the projection in $\operatorname{Teich}(S)$ of a trajectory of an non-autonomous Hamiltonian flow on $T^{*}\operatorname{Teich}(S)$: we call this flow the Moncrief flow, and the curves  the Moncrief lines. It is natural to ask whether the curve defined by the $CMC$ time converges when
time goes to $0$ to the point, in the Thurston
boundary of the Teichmüller space $\operatorname{Teich}(S)$, corresponding to the measured geodesic lamination, and when time goes to $+\infty$, to the hyperbolic
structure $\mathbb{H}^{2}/\pi_{1}(M)$. One also can ask this question for the curve defined by the $k$-time.

\section{Backgrounds and statement of results}
\subsection{Generalities on geometric metric spaces}
Let $(X,d)$ be a metric space. The length $L_{d}(\alpha)$ of a path $\alpha:[a,b]\rightarrow X$ is defined to be the supremum, on finite subdivion of $[a,b]$, of $\sum d(\alpha(t_{i}),\alpha(t_{i+1}))$. 
 The length  distance $d_{L}(x,y)$ between two points $x$ and $y$ is the infimum of the length of paths  joining $x$ and $y$. The metric space $(X,d_{L})$ is then called a length metric space. A path $\alpha$ joining two points $x$ and $y$ is a geodesic of the length metric space $(X,d_{L})$ if $L_{d}(\alpha)=d_{L}(x,y)$. A length metric space such that every two points are joined by a geodesic is called geodesic metric space. 

Let $(X,d)$ be a geodesic metric space. Let $\Delta(x,y,z)$ be a geodesic triangle in $X$. A comparison triangle of $\Delta(x,y,z)$ in the model space $(\mathbb{R}^{2},d_{\mathbb{R}^{2}})$ is the unique (up to isometry) triangle $\bar{\Delta}(\bar{x},\bar{y},\bar{z})$ of $(\mathbb{R}^{2},d_{\mathbb{R}^{2}})$ such that $d(x,y)=d_{euc}(\bar{x},\bar{y})$, $d(y,z)=d_{euc}(\bar{y},\bar{z})$ and $d(x,z)=d_{euc}(\bar{x},\bar{z})$. The comparison map  from $\Delta(x,y,z)$ to $\bar{\Delta}(\bar{x},\bar{y},\bar{z})$ is the unique map which sends the points $x$, $y$, $z$ to the points $\bar{x}$, $\bar{y}$, $\bar{z}$ and the geodesic segments $[x,y]$, $[x,z]$, $[y,z]$ to the geodesic segments  $[\bar{x},\bar{y}]$, $[\bar{x},\bar{z}]$, $[\bar{y},\bar{z}]$.
\begin{defi}
A geodesic metric space $(X,d)$ is $\operatorname{CAT}(0)$ if every comparison map is $1$-Lipschitz. 
\end{defi} 

A length metric space $(X,d)$ is said to possess the \textit{approximative midpoints property} if: for every $x$, $y$ in $X$ and  $\epsilon>0$ there exists $z$ in  $X$ such that $d(x,z)\leq \frac{1}{2}d(x,y)+\epsilon$ and $d(y,z)\leq \frac{1}{2}d(x,y)+\epsilon$. The length metric space $X$ satisfies the \textit{$\operatorname{CAT}(0)$ $4$-points condition} if for any $4$-tuple of points $(x_{1},y_{1},x_{2},y_{2})$ there exists a $4$-tuple of points $(\bar{x}_{1},\bar{y}_{1},\bar{x}_{2},\bar{y}_{2})$  in $\mathbb{R}^{2}$ such that: $d(x_{i},y_{j})=d(\bar{x}_{i},\bar{y}_{j})$ for $i,j\in\left\{1,2\right\}$, and $d(x_{1},x_{2})\leq d(\bar{x}_{1},\bar{x}_{2})$, $d(y_{1},y_{2})\leq d(\bar{y}_{1},\bar{y}_{2})$. Note that a $\operatorname{CAT}(0)$ metric space  satisfies the \textit{$\operatorname{CAT}(0)$ $4$-points condition} and have the \textit{approximative midpoints property}. The converse is true in the complete case:

\begin{prop}{\em (\cite[Proposition~II.1.11]{bridson})}
\label{prop.cat000000}
Let $(X,d)$ be a complete metric space. The following conditions are equivalent: 
\begin{itemize}
\item1) $X$ is a $\operatorname{CAT}(0)$ metric space;
\item2) $X$ possesses the \textit{approximative midpoints property} and satisfies the \textit{$\operatorname{CAT}(0)$ $4$-points condition}.
\end{itemize}
\end{prop}

A geodesic metric space $(X,d)$ is a real tree if any two points are joined by a unique path. Clearly a real tree is a $\operatorname{CAT}(0)$ metric space. An important example of real tree is the one given by a measured geodesic lamination (see for example \cite{MORGAN}, \cite{Otal}).
\subsection{Flat Regular domain, initial Singularity and Horizon}
Let $\mathbb{R}^{1,n}$ be the Minkowski space. An hyperplane $P$ is said to be lightlike if it is orthogonal to a lightlike direction. Let $\mathfrak{P}$ be the space of all lightlike hyperplanes in $\mathbb{R}^{1,n}$. Let $\Lambda$ be a closed subset of $\mathfrak{P}$ and consider $\Omega:=\bigcap_{P\in\Lambda} I^{+}(P)$. By \cite{barbot1}, the subset $\Omega$ is an open convex domain of $\mathbb{R}^{1,n}$. It is non empty as soon as $\Lambda$ is compact.
If $\Lambda$ contains more than two elements, then the open convex domain $\Omega$, if not empty, is called a future complete regular domain. In the same way one can define a past complete regular domain.

Let $\Omega$ be a future complete regular domain. The boundary $\partial\Omega$ of $\Omega$ is convex. By \cite{bonsante1}, it is the graph of a $1$-Lipschitz convex function $f:\mathbb{R}^{n}\rightarrow \mathbb{R}$.

Let $\mathfrak{L}$ be the set of Lipschitz curves contained in $\partial\Omega$. For every $\alpha\in \mathfrak{L}$, consider $l(\alpha):=\int\left|\dot{\alpha}(t)\right|dt$ the Lorentzian length of $\alpha$. Let $d$ be the pseudo-distance defined on $\partial\Omega$ by: $$d_{\partial \Omega}(p,q)=\inf\left\{l(\alpha),\mbox{~where~} \alpha \mbox{~is a curve in $\mathfrak{L}$ joining p and q~}\right\}.$$
The cleaning $(\partial \Omega/\sim, \bar{d}_{\partial \Omega})$ i.e the quotient of the pseudo metric space $(\partial \Omega^{+}, d_{\partial \Omega})$ by the equivalence relation $p\sim q$ if and only if $d_{\partial \Omega}(p,q)=0$, is a length metric space (see for instance \cite[Corollaire~2.2.14]{mehdi2} ).
\begin{defi}
The metric space $(\partial \Omega/\sim, \bar{d}_{\partial \Omega})$ is the Horizon associated to $\Omega$.
\end{defi}

An hyperplane $P$ is a support hyperplane of $\Omega$ if $\Omega\subset J^{+}(P)$. Note that if $\Omega$ admits two lightlike support hyperplanes then it admits a spacelike support hyperplane. Let $\Sigma$ be the set of points $p\in\partial \Omega$ such that $\Omega$ have a  spacelike support hyperplane passing through $p$. By a result of Bonsante (see \cite[Proposition~7.8]{bonsante1} ), the restriction of the pseudo-distance  $ d_{\partial \Omega}$ to $\Sigma$ is a distance denoted by $d_{\Sigma}$. 
\begin{defi}
The metric space $(\Sigma,d_{\Sigma})$ is  the initial Singularity associated to $\Omega$.
\end{defi}
\begin{bei}
The solid cone $\mathbf{C}$ is a typical example of regular domain. In this case the metric spaces $(\partial \mathbf{C}/\sim, \bar{d}_{\partial \mathbf{C}})$ and $(\Sigma,d_{\Sigma})$ are identified with the trivial metric space $(\{0\},d=0)$.
\end{bei}

In \cite{bonsante1}, Bonsante shows that to each point $p$ in $\Omega$ corresponds a unique point $r(p)$ in $\partial \Omega$ realizing the cosmological time i.e such that $T_{cos}(p)=\left|p-r(p)\right|^{2}$. He proved also that the application $r:\Omega\rightarrow \partial \Omega$, called retraction map, is continuous and that $r(\Omega)=\Sigma$. Moreover, the cosmological time $T_{cos}$ of $\Omega$ is a $C^{1,1}$ regular Cauchy time whose Lorentzian gradient is given by $N_{p}=-\nabla_{p} T_{cos}=\frac{1}{T_{cos}(p)}(p-r(p))$. By \cite[Lemma~4.15]{bonsante1} and \cite[Lemma~3.12, Corollary~4.5]{bonsante1}, the normal application $N:\Omega\rightarrow \mathbb{H}^{n}$, when restricted to each cosmological level $S^{T_{cos}}_{a}$, is a surjective proper function. Every point $p$ in $\Omega$ can be decomposed as $$p=r(p)+T_{cos}(p)N_{p}.$$
Actually all this remain true in any future complete convex domain of $\mathbb{R}^{1,n}$.

\subsection{Flat space-times and flat regular domains}
Let $\Gamma$ be a torsion free uniform lattice of $SO^{+}(1,n)$. A cocycle of $\Gamma$ is an application $\tau:\Gamma\rightarrow \mathbb{R}^{1,n}$ such that $\tau(\gamma_{1}.\gamma_{2})=\gamma_{1}\tau(\gamma_{2})+\tau(\gamma_{1})$. An affine  deformation of $\Gamma$ associated to $\tau$ is the morphism $\rho_{\tau}:\Gamma\rightarrow SO^{+}(1,n)\ltimes\mathbb{R}^{1,n}$ defined by  $\rho_{\tau}(\gamma).x=\gamma.x+\tau(\gamma)$ for every $\gamma\in \Gamma$ and $x\in\mathbb{R}^{1,n}$. By a result of Bonsante \cite{bonsante1}, to every affine deformation of $\Gamma$ corresponds a unique (up to reorientation) flat future complete regular domain $\Omega$ on which $\Gamma_{\tau}=\rho_{\tau}(\Gamma)$ acts freely properly discontinuously. In this case, the cosmological normal application $N$ and the retraction map $r$ of $\Omega$ are equivariant under the action of $\Gamma$. This means that $N_{\gamma_{\tau}.p}=\gamma.N_{p}$ and $r(\gamma_{\tau}.p)=\gamma_{\tau}.r(p)$ for every $p$ in $\Omega$ and $\gamma$ in $\Gamma$. The space-time $M_{[\tau]}:=\Omega/\Gamma_{\tau}$ is then called a standard flat space-time. In the special case of the trivial cocycle the space-time $M_{[0]}:=\mathbf{C}/\Gamma$ is the static flat space-time.

A future complete $MGHC$ flat space-time $M$ is said to be non elementary if $L(\pi_{1}(M))$ is a non elementary subgroup of $SO^{+}(1,n)$, where $L:\pi_{1}(M)\rightarrow SO^{+}(1,n)$ is the linear part of the holonomy morphism $\rho:\pi_{1}(M)\rightarrow SO^{+}(1,n)\ltimes\mathbb{R}^{1,n}$ of $M$. The following theorem gives a full classification of $MGHC$ flat non elementary space-times.
\begin{theo}{\em (\cite[Theorem~4.11]{barbot1})}
Every future complete $MGHC$ flat non elementary space-time $M$ is up to finite cover the quotient of a future  complete regular domain by a discrete subgroup of $SO^{+}(1,n)\ltimes\mathbb{R}^{1,n}$.
\end{theo}

\subsection{(2+1)-de Sitter space-times} 
Let $S$ be a simply connected Möbius manifold. That is a manifold equipped with a $(G,X)$-structure, where $G=O^{+}(1,n)$ and $X=\mathbb{S}^{n}$ is the Riemannian sphere. A  Möbius manifold is elliptic (respectively parabolic) if it is conformally equivalent to $\mathbb{S}^{n}$ (respectively  $\mathbb{S}^{n}$ minus a point). A non elliptic neither parabolic Möbius manifold is called hyperbolic Möbius manifold.

Let $d:S\rightarrow \mathbb{S}^{n}$ be a developing map of $S$. A round ball of $S$ is an open convex set $U$ of $S$ on which $d$ is an homeomorphism. It is said to be proper if $d(\bar{U})$ is a closed round ball of $\mathbb{S}^{n}$. Let $B(S)$ be the space of proper round ball of $S$. By a result of \cite{BBZ1}, there is a natural topology on $B(S)$ making it locally homeomorphic to $\mathbb{D}\mathbb{S}_{n+1}$. By \cite{BBZ1}, the space $B(S)$ endowed with the pull back metric of $\mathbb{D}\mathbb{S}_{n+1}$ is a simply connected future complete globally hyperbolic locally de Sitter space-time called $dS$-standard space-time.

In general $B(S)$ is not isometric to a part of $\mathbb{D}\mathbb{S}_{n+1}$. However, there are some regions in $B(S)$ which embedd isometrically in $\mathbb{D}\mathbb{S}_{n+1}$. Indeed, let $x$ in $S$ and let $U(x)$ be the union of all round containing $x$. Then by \cite{BBZ1}, the $dS$-standard spacetime $B(U(x))$ is isometric to an open domain of  $\mathbb{D}\mathbb{S}_{n+1}$. Moreover, for every proper round ball $V$ containing $x$, the causal past of $V$ in $B(S)$ is contained in $B(U(x))$. 

In the case of $dS$-standard space-time of hyperbolic type, the cosmological time  is regular (see \cite{BBZ1}). One can attach to each hyperbolic type $dS$-standard space-time $B(S)$ a past boundary $\partial B(S)$, which can be seen locally as a convex hypersurface of $\mathbb{D}\mathbb{S}_{n+1}$. Moreover, to every point $p$ in $B(S)$ corresponds a unique point $r(p)$ on $\partial B(S)$ realizing the cosmological time. Actually the point $r(p)$ is the limit point in $B(S)\cup\partial B(S)$ of the past timelike geodesique starting at $p$ with initial velocity $-N_{p}$, where $N_{p}$ is the future oriented cosmological normal vector at $p$. The application $N$ is the cosmological normal application and $r$ is the retraction map.

We have the following classification theorem:
\begin{theo}{\em (\cite[Theorem~1.1]{scannell1})}
Every $MGHC$ future complete de Sitter space-time is, up to reorientation, the quotient of a standard $dS$ space-time by a free torsion discret subgroup of $SO^{+}(1,n+1)$.
\end{theo}  

\begin{defi}
Let $M$ be a differentiable manifold endowed with two Lorentzian metrics $g$ and $\mathfrak{g}$. Let $\xi$ be a vector fields everywhere non zero. The Lorentzian metric $\mathfrak{g}$ is obtained by a Wick rotation from the Lorentzian metric $g$ along the vector fields $\xi$ if: 
\begin{itemize}
\item 1) For every $p$ in $M$, the sub-spaces $g$-orthogonal and $\mathfrak{g}$-orthogonal to $\xi_{p}$ are the same;

\item 2) there exists a positive function $f$ such that $\mathfrak{g}=f g$ on the sub-space spanned by $\xi_{p}$;

\item 3) There exists a positive function $h$ such that : $\mathfrak{g}=hg$ on $ \xi_{p}^{\perp}$.
\end{itemize}
\end{defi}

Let $\Omega$ be a flat future complete regular domain of dimension $2+1$. Consider $\Omega_{1}$ the past in $\Omega$ of the cosmological level $S^{T_{cos}}_{1}$ and $g$  its induced Lorentzian metric. By \cite{benebonsan1}, there exists a $C^{1}$ local diffeomorphism $\Hat{D}:\Omega_{1}\rightarrow \mathbb{D}\mathbb{S}_{3}$ such that the pullback by $\Hat{D}$ of the de Sitter metric is the Lorentzian metric $\mathfrak{g}$ obtained from $g$ by a Wick rotation along the cosmological gradient with
$\mathfrak{g}=\frac{1}{(1-T_{cos}^{2})^{2}}g$ on $\mathbb{R}\xi_{T_{cos}}$ and $\mathfrak{g}=\frac{1}{1-T_{cos}^{2}}g$ on $\left\langle \xi_{T_{cos}}\right\rangle^{\perp}$. The space $(\Omega_{1},\mathfrak{g})$ is a $dS$-standard spacetime of hyperbolic type i.e associated to some hyperbolic projective structure (given also by the canonical Wick rotation) on $S^{T_{cos}}_{1}$. In fact, this Wick rotation provides us a one to one correspondence between standard $2+1$ de Sitter space-times of hyperbolic type and flat future complete regular domains of dimension $2+1$. Moreover, this construction can be done in an equivariant way giving hence a one to one correspondence between future complete flat $MGHC$ non elementary space-times of dimension $2+1$ and future complete $MGHC$ de Sitter space-times of hyperbolic type of dimension $2+1$.
\begin{prop}{\em (\cite[Proposition~5.2.1]{benebonsan1})}
The cosmological time $\mathcal{T}_{cos}$ of $(\Omega_{1},\mathfrak{g})$ is a Cauchy time. Moreover, $$\mathcal{T}_{cos}=\operatorname{argth}(T_{cos}),$$ where $T_{cos}$ is the cosmological time of $(\Omega_{1},g)$.
\end{prop}

\subsection{(2+1)-anti de Sitter space-times} 
Let $M$ be a $MGHC$ anti de Sitter space-time of dimension $n+1$. By \cite{mess1}, \cite{BBZ1}, the universal cover $\widetilde{M}$ of $M$ is isometric to an open convex domain, called regular domain, of the anti de Sitter space. Denote by $\widetilde{M}_{-}$ the tight past of $\widetilde{M}$ i.e the strict past in  $\widetilde{M}$ of the cosmological level $S^{T_{cos}}_{\frac{\pi}{2}}$.

By \cite{BBZ1}, the cosmological time of a $\widetilde{M}_{-}$ is regular. One can attach a past boundary $\partial \widetilde{M}_{-}$ to $\widetilde{M}_{-}$ which can be seen as a convex hypersurface of $\mathbb{A}\mathbb{D}\mathbb{S}_{n+1}$. Moreover, to every point $p$ in $\widetilde{M}_{-}$ corresponds a unique point $r(p)$ on $\partial \widetilde{M}_{-}$ realizing the cosmological time. the point $r(p)$ is the limit point in $\widetilde{M}_{-}\cup\partial \widetilde{M}_{-}$ of the past timelike geodesique starting at $p$ with initial velocity $-N_{p}$, where $N_{p}$ is the future oriented cosmological normal vector at $p$. The application $N$ is the cosmological normal application and $r$ is the retraction map.

Let $\Omega$ be a flat future complete regular domain of dimension $2+1$ and let $g$ be its induced Lorentzian metric. By \cite{benebonsan1} there exists a $C^{1}$ local diffeomorphism $\Hat{D}:\Omega\rightarrow \mathbb{A}\mathbb{D}\mathbb{S}_{3}$ such that the pullback by $\Hat{D}$ of the anti de Sitter metric is the Lorentzian metric $\mathfrak{g}$ obtained from $g$ by a Wick rotation along the cosmological gradient with
$\mathfrak{g}=\frac{1}{(1+T_{cos}^{2})^{2}}g$ on $\mathbb{R}\xi_{T_{cos}}$ and $\mathfrak{g}=\frac{1}{(1+T_{cos}^{2})}g$ on $\left\langle \xi_{T_{cos}}\right\rangle^{\perp}$. 
In fact $(\Omega,\mathfrak{g})$ is the tight past region of its maximal anti de Sitter extension. Moreover, this Wick rotation provide us a one to one correspondence between $2+1$ anti de Sitter regular domains and flat future complete regular domains of dimension $2+1$. This construction can be done in an equivariant way giving hence a one to one correspondence between future complete flat $MGHC$ non elementary space-times of dimension $2+1$ and future complete $MGHC$ anti de Sitter space-times of dimension $2+1$.
\begin{prop}{\em (\cite[Proposition~6.2.2]{benebonsan1})}
The cosmological time $\mathcal{T}_{cos}$ of $(\Omega,\mathfrak{g})$ is a Cauchy time. Moreover, $$\mathcal{T}_{cos}=\arctan(T_{cos}),$$ where $T_{cos}$ is the cosmological time of $(\Omega,g)$.
\end{prop}

\subsection{Statement of results}
Let $M$ be a future complete $MGHC$ space-time of constant curvature. Let $T:\widetilde{M}\rightarrow \mathbb{R}$ be a $\pi_{1}(M)$-invariant quasi-concave Cauchy time. Up to reparametrization we can suppose that $T$ takes its values in $\mathbb{R}^{*}_{+}$. Consider the family of $\pi_{1}(M)$-inveriant metric spaces $(\pi_{1}(M),S^{T}_{a},d^{T}_{a})_{a\in\mathbb{R}^{*}_{+}}$ associated to $T$. Let $\gamma\in \pi_{1}(M)$ and $a>0$, denote by $l_{a}^{T}(\gamma):=\inf_{x\in S^{T}_{a}}d_{a}^{T}(x,\gamma.x)$ the marked spectrum of $d^{T}_{a}$.

Benedetti and Guadignini \cite{beneguadinini} conjectured that:
\begin{conj}
Let $M$ be a future complete $MGHC$ non elementary flat space-time of dimension $2+1$ and let $T_{cmc}$ be the associated $CMC$ time. Then:
\begin{itemize}
\item 1) $\lim_{a\rightarrow 0} l_{a}^{T}(\gamma)=l_{\Sigma}(\gamma)$
\item 2) $\lim_{a\rightarrow +\infty} a^{-1}l_{a}^{T}(\gamma)=l_{\mathbb{H}^{2}}(\gamma)$
\end{itemize}
\end{conj}
Andersson \cite{anderson3}  gives a positive answer to the first part of this conjecture in the case of simplicial space-times. In \cite{mehdi1} we studied the past asymptotic behavior of  quasi-concave Cauchy times in a $2+1$ flat space-times. We gave in particular a positive answer to the first part of the Benedetti-Guadignini conjecture.
\begin{theo}{\em (\cite[Theorem~1.1]{mehdi1})}
\label{theomehdi1}
Let $M$ be a future complete $MGHC$ non elementary flat space-time of dimension $2+1$. Let $T$ be a $C^{2}$ quasi-concave Cauchy time function on $\widetilde{M}$. Then the levels $(\pi_{1}(M),S^{T}_{a},d^{T}_{a})_{a\in \mathbb{R}^{*}_{+}}$ converge in the Gromov equivariant topology, when $a$ goes to $0$, to the real tree dual to the measured geodesic lamination associated to $M$. In particular this limit does not depend on the time function $T$.  
\end{theo}

Our two first results concern the asymptotic behavior in the flat $n+1$ dimensional case. In dimension bigger than $3$, the situation is more complicated. The initial singularity is no longer a real tree in general (see \cite{bonsante1}). However, we have the following partial result which is a generalization of \ref{theomehdi1} to the $n+1$-dimensional flat case:
\begin{theo}
\label{mehditheorem2}
Let $M\simeq \Omega/\Gamma$ be a future complete $MGHC$ flat non elementary space-time of dimension $n+1$, where $\Omega$ is a future complete regular domain and $\Gamma$ a discrete subgroup of $SO^{+}(1,n)\ltimes \mathbb{R}^{1,n}$. Denote by $(\Sigma,d_{\Sigma})$, $(\Sigma^{\bigstar}, d^{\bigstar}_{\Sigma})$ and $(\partial \Omega/\sim, \bar{d}_{\partial \Omega})$ respectively the initial Singularity, its completion and the Horizon associated to $M$. Let $T$ be a $C^{2}$ quasi-concave Cauchy time  on $\widetilde{M}$. Then:
\begin{itemize}
\item $(\Sigma,d_{\Sigma})$ embeds isometrically in $(\partial \Omega/\sim, \bar{d}_{\partial \Omega})$ which embeds isometrically in $(\Sigma^{\bigstar}, d^{\bigstar}_{\Sigma})$;
\item For every $x$ and $y$ in $\Sigma$ there exists a unique geodesic $\alpha$ of $(\partial \Omega^{+}/\sim, \bar{d}_{\partial \Omega^{+}})$ joining  $x$ and $y$;
\item The metric space $(\Sigma^{\bigstar}, d^{\bigstar}_{\Sigma})$ is a $CAT(0)$ metric space;
\item The levels $(\pi_{1}(M),S^{T}_{a},d^{T}_{a})_{a\in \mathbb{R}^{*}_{+}}$ converge in the Gromov equivariant topology, when $a$ goes to  $0$, to the completion metric space $(\pi_{1}(M),\Sigma^{\bigstar}, d^{\bigstar}_{\Sigma})$ of the initial singularity $(\pi_{1}(M),\Sigma,d_{\Sigma})$. In particular the limit does not depend on the time function $T$.
\end{itemize}
\end{theo}

Near the infinity we obtain the following result:
\begin{theo}
\label{mehditheorem3}
Let $M$ be a future complete $MGHC$ flat non elementary space-time of dimension $n+1$. Then,
\begin{itemize}
\item1) There exists a constant $C(M)$ such that for every $C'>C$ and every quasi-concave  Cauchy time $T$ on $\widetilde{M}$, the renormalized $T$-levels $(\pi_{1}(M),S^{T}_{a},(\sup_{S^{T}_{a}} T_{cos})^{-1}d^{T}_{a})_{a\in \mathbb{R}^{*}_{+}}$ are, for $a$  big enough, $C'$-quasi-isometric to $(\pi_{1}(M),\mathbb{H}^{n},d_{\mathbb{H}^{n}})$. In particular all the limit points, for the Gromov equivariant topology, of $(\pi_{1}(M),S^{T}_{a},(\sup_{S^{T}_{a}} T_{cos})^{-1}d^{T}_{a})_{a\in \mathbb{R}^{*}_{+}}$ are $C$-bi-Lipschitz to $(\pi_{1}(M),\mathbb{H}^{n},d_{\mathbb{H}^{n}})$;
\item2) In dimension $2+1$, the renormalized $CMC$-levels (respectively $k$-levels) converge in the Gromov equivariant topology, when time goes to $+\infty$, to $(\pi_{1}(M),\mathbb{H}^{n},d_{\mathbb{H}^{n}})$.
\end{itemize}
\end{theo}
\begin{rmq}
1) The constant $C$ in Theorem \ref{mehditheorem3} depends actually only on $\pi_{1}(M)$.\\
2) In fact Theorem \ref{mehditheorem3} is the best result we can get in this generality. Indeed, in a static flat space-time $(\Gamma,\mathbf{C})$, consider a $\Gamma$-invariant complete convex surface $S$ different than the cosmological ones. The family $(a S)_{a>0}$ constitutes a foliation of $\mathbf{C}$. The associated renormalized family of metric spaces converges in the Gromov equivariant topology, when $a$ goes to $+\infty$, to $(\Gamma,S,d_{S})$.
\end{rmq}

Now focus on the  $2+1$ dimensional case. In this article we obtain the analogue of Theorem \ref{theomehdi1} in the de Sitter and anti de Sitter cases. More precisely:
\begin{theo}
\label{Theomehdidesianti}
Let $M$ be $MGHC$ de Sitter (or anti de Sitter) space-time of dimension $2+1$. Let $T$ be a $C^{2}$ quasi-concave Cauchy time on $\widetilde{M}$. Then the levels $(\pi_{1}(M),S^{T}_{a},d^{T}_{a})_{a\in \mathbb{R}^{*}_{+}}$ converge in the Gromov equivariant topology, when $a$ goes to $0$, to the real tree dual to the measured geodesic lamination associated to $M$. In particular this limit does not depend on the time function $T$.  
\end{theo}

Now look to the asymptotic behavior in the Teichmüller space. Our fourth result concern the future behavior of the curve associated to the $k$-time. Let $S$ be a closed hyperbolic surface and let $(\lambda, \mu)$ be a measured geodesic lamination on $S$. Let $M$ be the $MGHC$ space-time of constant curvature associated to $(\lambda, \mu)$.
\begin{theo}
\label{theormehditeich1}
Let $T_{k}$ and $T_{cmc}$ be respectively the $k$-time and the $CMC$ time of $M$. Then,
\begin{itemize}
\item In the flat case: the curves $([g^{T_{k}}_{a}])_{a>0}$ and $([g^{T_{k}}_{a}])_{a>0}$ in the Teichmüller space $\operatorname{Teich}(S)$ of $S$ converge, when time goes to $+\infty$, to the hyperbolic structure of $S$.
\item In the de Sitter case: The curve $([g^{T_{k}}_{a}])_{a>0}$ in the Teichmüller space $\operatorname{Teich}(S)$ of $S$ stays at a bounded Teichmüller distance, when time goes to $+\infty$, from the grafting metric $\operatorname{gra}_{\lambda}(S)$.
\end{itemize}
\end{theo}


\noindent{\bf Acknowledgements.}
I want to thank Thierry Barbot for his considerable support and for many helpful discussions.

\section{Quasi-concave times and there expansive Character}
Let $M$ be a $MGHC$ space-time of constant curvature. Let $S$ be a $C^{2}$ complete $\pi_{1}(M)$-invariant spacelike hypersurface of $\widetilde{M}$. Let $\Pi_{S}$ be its second fundamental form. Recall that the mean curvature $H_{S}$ at a point $p$ of $S$ is defined by $H_{S}=\frac{tr(\Pi)}{n}$ i.e $H_{S}=\frac{\lambda_{1}+\lambda_{2}+...+\lambda_{n}}{n}$, where  $\lambda_{i}$ are the principal curvatures of $S$. Recall that in the case of dimension $2$, the Gauss curvature $k_{S}$ at a point $p$ of $S$ is defined by $k_{S}=-det(\Pi)$ i.e $k_{S}=-\lambda_{1}\lambda_{2}$.
\begin{defi}
The hypersurface $S$ is said to be quasi-concave if its second fundamental form is positive-definite.
\end{defi}
The convexity of $S$ is equivalent to the geodesic convexity of $J^{+}(S)$. Thus using this last characterisation one can generalise the notion of convexity to non smooth hypersurfaces.

A $\pi_{1}(M)$-invariant Cauchy time function $T:\widetilde{M}\rightarrow\mathbb{R}^{*}_{+}$ is quasi-concave if its levels are convex. 
The cosmological time, the $CMC$ time and the $K$ time provide us important examples of quasi-concave times.
\begin{defi}
The cosmological time $T_{cos}$ is defined by: $T_{cos}(p)=\sup_{\alpha}\int\sqrt{-\left|\dot{\alpha}(s)\right|^{2}}$ where the supremum is taken over all the past causal curves starting at p.
\end{defi}
In the flat case the cosmological time is a concave (and hence quasi-concave) Cauchy time (see \cite{bonsante1}). By \cite{scannell1}, \cite{BBZ1} the cosmological time is a regular quasi-concave time in the de Sitter case. In the anti de Sitter case it false to be quasi-concave. However, by \cite{BBZ1} the cosmological levels are convex near the initial singularity.
\begin{defi}
The $CMC$ time is a $\pi_{1}(M)$-invariant  Cauchy time $T:\widetilde{M}\rightarrow \mathbb{R}$ such that every level $T^{-1}(t)$, if not empty, is of constant mean curvature $t$. 
\end{defi}
The existence and uniqueness of such time was studied in \cite{anderson2}, \cite{anderson6}, \cite{bbz4}, \cite{bbz5},\cite{BBZ1}. In the flat case and by a result of Treibergs \cite{triberg} the $CMC$ time is quasi-concave. It is no more true in the anti de Sitter case. Unfortunately we don't now if it is the case in the de Sitter case.

In the flat case, the $CMC$ time takes its values over $\mathbb{R}^{*}_{-}$. Up to the reparametrization $b\mapsto -\frac{1}{b}$, we will consider that the $CMC$ time takes its values in $\mathbb{R}^{*}_{+}$. In other words: for every $b>0$, the $CMC$ level $S^{T_{cmc}}_{b}$ is of constant mean curvature $-\frac{1}{b}$. 

\begin{defi}
Suppose that $M$ is of the dimension $2+1$. The $k$ time is a Cauchy time $T:\widetilde{M}\rightarrow \mathbb{R}$ such that every level $T^{-1}(t)$, if not empty, is of constant Gauss curvature $t$.
\end{defi}
Barbot, B\'eguin and Zeghib \cite{BBZ2} proved the existence and uniqueness of such time in the flat and de Sitter case. In the anti de Sitter case there is no  globally defined $k$-time. However, the two connected components of the convex core admit a unique $k$-time. By definition, the $k$-time is quasi-concave.

In the flat and the anti de Sitter cases, the $k$-time is defined over $\mathbb{R}^{*}_{-}$. Up to the reparametrization $b\mapsto \sqrt{-b^{-1}}$, we will consider that the $k$-time takes its values over $\mathbb{R}^{*}_{+}$. In the de Sitter case, the $k$-time is defined over $]-\infty,-1[$. So we will consider it defined over $\mathbb{R}^{*}_{+}$ up to the reparametrization $b\mapsto \sqrt{-(b+1)^{-1}}$.

Let $T:\widetilde{M}\rightarrow\mathbb{R}^{*}_{+}$ be a $\pi_{1}(M)$-invariant $C^{2}$ quasi-concave Cauchy time. Denote by $\xi_{T}=\frac{\nabla T}{\left|\nabla T\right|^{2}}$, where $\nabla T$ is the Lorentzian  gradient of $T$ and let $\Phi^{t}_{T}$ be the corresponding flow generated by $\xi_{T}$. Denote by $S^{T}_{1}$ the level set $T^{-1}(1)$ of $T$. 
\begin{prop}
\label{propositionmehdi1}
Let $\alpha :[a, b]\rightarrow\widetilde{M}$ be a spacelike curve contained in the past of $S^{T}_{1}$. Then the length of $\alpha$ is less than the length of $\alpha_{1}$ where $\alpha_{1}(s)=\Phi_{T}^{1-T(\alpha(s))}(\alpha(s))$ is the projection of $\alpha$ on $S^{T}_{1}$ along the lines of $\Phi_{T}$.
\end{prop}
\begin{proof}
We proved this proposition in the 2+1 flat case \cite[Proposition~4.2]{mehdi1}. The proof does not use the fact that space-time is flat of dimension $2+1$ and remains true in our case (see \cite[Remark~1.2]{mehdi1}).
\end{proof}
\begin{rmq}
Even if in Proposition \ref{propositionmehdi1} we restrict ourselves to $C^{2}$ quasi-concave times, one can prove analogue Propositions for the cosmological time, which is just $C^{1,1}$, in the Sitter and anti de Sitter cases (see Remark \ref{jeudivac1} and Remark \ref{jeudivac2}). 
\end{rmq}

\section{Quasi-concave times versus Cosmological time}
Let $M$  be a positive constant curvature $MGHC$ space-time of dimension $n+1$ and let $T_{cos}$ be the cosmological time of $\widetilde{M}$. The purpose of this section is to highlight the comparability between the cosmological time and the other quasi-concave times.

\subsection{The flat case}
Let us start with the following proposition which gives an estimation on the cosmological barriers in the flat $n+1$ dimensional case. 
\begin{prop}
\label{propositionmehdi2}
Let $M\simeq \Omega/\Gamma_{\tau}$ be a standard flat space-time, where $\Omega$ is a future complete flat regular domain, $\Gamma$ a torsion free uniform lattice of $SO^{+}(1,n)$ and $\Gamma_{\tau}$ its affine deformation in $SO^{+}(1,n)\ltimes \mathbb{R}^{1,n}$. Let $S$ be a convex complete $\Gamma$-invariant Cauchy hypersurface of $\Omega$. There is a constant  $C$ depending only on $\Gamma$ such that for every $C'>C$  $$\frac{\sup_{S}T_{cos}}{\inf_{S}T_{cos}}\leq C',$$ for $\inf_{S}T_{cos}$ big enough.
\end{prop}
\begin{proof}
Fix an origin of the Minkowski space $\mathbb{R}^{1,n}$. Let $N$ and $r$ be respectively the normal application and the retraction map of $\Omega$. 

For simplicity denote by $a=\sup_{S}T_{cos}$ and by $b=\inf_{S}T_{cos}$. Let $F\subset \mathbb{H}^{n}$ be a compact fundamental domain for the action of $\Gamma$ on $\mathbb{H}^{n}$. Note that $F'=r(N^{-1}(F'))$ is a fundamental domain for the action of $\Gamma_{\tau}$ on $\Sigma$. The closure of $F'$ in $\mathbb{R}^{1,n}$ is compact. Denote then by $C_{1}=\sup_{F'\times F'\times F}\left|\left\langle r_{1}-r_{2}, n\right\rangle\right|$. 

Now let $p\in S$ such that $T_{cos}(p)=a$. Up to isometry we can suppose that $N_{p}\in F$ and $r(p)\in F'$. The convexity of $S$ implies that the tangent hyperplane $P_{p}$ to $S$ at $p$ is the tangent hyperplane to $S^{T_{cos}}_{a}$ at $p$. Thus for every $\gamma$ in $\Gamma$, $\gamma_{\tau}.P_{p}$ is the tangent hyperplane of $S$ and $S^{T_{cos}}_{a}$ at $\gamma_{\tau}.p$. Hence, we obtain that for every $x$ in $S$ and every $\gamma$ in $\Gamma$: $$\left\langle \gamma_{\tau}p-x,\gamma.N_{p}\right\rangle\geq 0$$  But  $\gamma_{\tau}p=\gamma.p+\tau(\gamma)$, $x=r(x)+T_{cos}(x)N_{x}$ and $p=r(p)+T_{cos}(p)N_{p}$, so $$T_{cos}(x)\left\langle N_{x},\gamma.N_{p}\right\rangle\leq \left\langle p,N_{p}\right\rangle-\left\langle \gamma^{-1}r(x)+\tau(\gamma^{-1}),N_{p}\right\rangle$$ Therefore
$$\left|\frac{1}{\left\langle \gamma^{-1}.N_{x}, N_{p}\right\rangle}\right|-\frac{1}{a}\frac{\left\langle r(p)-r(\gamma_{\tau}^{-1}x),N_{p}\right\rangle}{\left|\left\langle \gamma^{-1}.N_{x}, N_{p}\right\rangle\right|}\leq \frac{T_{cos}(x)}{a}$$

On the other hand, for every $x$ in $S$, there exists a $\gamma_{x}$ in $\Gamma$ such that $\gamma_{x}^{-1}.N_{x}\in F$ and $r((\gamma_{x})_{\tau}^{-1}x)\in F'$.  Thus,

$$\frac{1}{C}-\frac{1}{a}C_{1}\leq\frac{T_{cos}(x)}{a},$$ where $C=\sup_{F\times F}\left|\left\langle n,n'\right\rangle\right|$

Since the last inequality is true for every $x$ in $S$, we obtain that
$$\frac{1}{C}-\frac{1}{a}C_{1}\leq\frac{b}{a}$$

And this finishes the proof.
\end{proof}

As a direct consequence of this proposition we obtain:
\begin{cor}
Let $T:\Omega\rightarrow ]0,+\infty[$ be a $\Gamma$-invariant quasi-concave Cauchy time. Then $$\lim_{b\rightarrow \infty} \frac{\sup_{S^{T}_{b}}T_{cos}}{\inf_{S^{T}_{b}}T_{cos}}\leq C$$
\end{cor}
\begin{rmq}
\label{remarmehdi3}
By a result of Andersson, Barbot, Béguin and Zeghib \cite{BBZ1} we have that in the particular case of the $CMC$ time : $\frac{\sup_{S^{T}_{b}}T_{cos}}{\inf_{S^{T}_{b}}T_{cos}}\leq n$ for every $b>0$. Moreover,  $$\frac{1}{n}\sup_{S^{T}_{b}}T_{cos}\leq b\leq \sup_{S^{T}_{b}}T_{cos}$$
\end{rmq}

\begin{prop}
\label{propmehdi3}
Let $M\simeq \Omega/\Gamma$ be a non elementary future complete $MGHC$ flat space-time of dimension $2+1$ and let $T_{k}:\Omega\rightarrow ]0,+\infty[$ be the $k$-time of $\Omega$. The cosmological time and the $k$-time are comparable near the infinity. Moreover $$\lim_{b\rightarrow +\infty} \frac{\inf_{S^{T_{k}}_{b}}T_{cos}}{b}=\lim_{b\rightarrow +\infty} \frac{\sup_{S^{T_{k}}_{b}}T_{cos}}{b}=1.$$
\end{prop}

For the proof we need the following Maximum Principle. 

\begin{lemm}
Let $S$ and $S'$ two spacelike hypersurfaces in a space-time $M$ such that $S'$ is in the future of $S$ and $S\cap S'\neq  \varnothing$. For every $p\in S\cap S'$ we have that the principal curvatures of $S$ at $p$ are bigger than the principal curvatures of $S'$ at $p$. In particular the Gauss curvature of $S$ is bigger than the Gauss curvature of $S'$.
\end{lemm}

\textbf{Proof of Proposition \ref{propmehdi3}}.
Let $S^{T_{k}}_{1}$ be the $k$-level of constant Gauss curvature $-1$. Let $H_{0}=\inf H_{S^{T_{k}}_{1}}$ and $H_{1}=\sup H_{S^{T_{k}}_{1}}$, where $H_{S^{T_{k}}_{1}}$ is the $CMC$ curvature of $S^{T_{k}}_{1}$.

Consider the $\Gamma$-invariant future complete convex domain $A:=J^{+}(S^{T_{k}}_{1})$. Denote respectively by $T'_{cos}$, $r'$ the associated cosmological time and retraction map. 
For every $b>1$, the $\Gamma$-invariant $k$-level $S^{T_{k}}_{b}$ is entirely contained in $A$. As the action of $\Gamma$ on $S^{T_{k}}_{b}$ is cocompact, the cosmological time $T'_{cos}$ of $A$ achieve its minimum on $S^{T_{k}}_{b}$. Let $p\in S^{T_{k}}_{b}$ such that $\inf_{S^{T_{k}}_{b}}T'_{cos}=T'_{cos}(p):=a$. By applying the Maximum Principle to the hypersurfaces $S^{T'_{cos}}_{a}$ and $S^{T_{k}}_{b}$ we get $$k_{S^{T'_{cos}}_{a}}(p)\geq -\frac{1}{b^{2}},$$ where $k_{S^{T'_{cos}}_{a}}$ is the Gauss curvature of $S^{T'_{cos}}_{a}$.

On the one hand we have $$k_{S^{T'_{cos}}_{a}}(p)=-\frac{1}{1-2H_{S^{T_{k}}_{1}}(r'(p))a+a^{2}}$$ Hence $$a\geq H_{0}+\sqrt{H_{1}^{2}-1+b^{2}}$$
But $$\inf_{S^{T_{k}}_{b}}T_{cos}\geq \inf_{S^{T'_{cos}}_{a}}T_{cos}\geq a$$
So $$\inf_{S^{T_{k}}_{b}}T_{cos}\geq H_{0}+\sqrt{H_{1}^{2}-1+b^{2}}$$

On the other hand and by applying the Maximum Principle to the hypersurfaces $S^{T_{k}}_{b}$ and $S^{T_{cos}}_{\sup_{S^{T_{k}}_{b}}T_{cos}}$ we get  $$\sup_{S^{T_{k}}_{b}}T_{cos}\leq b$$

Thus $$1\geq\frac{\sup_{S^{T_{k}}_{b}}T_{cos}}{b}\geq\frac{\inf_{S^{T_{k}}_{b}}T_{cos}}{b}\geq \frac{H_{0}}{b}+\sqrt{\frac{H_{1}^{2}}{b^{2}}-\frac{1}{b^{2}}+1}$$ which concludes the proof.
\fin
\begin{cor}
\label{cormehdi3}
We have:
$$\lim_{b\rightarrow +\infty} \frac{\inf_{S^{T_{cmc}}_{b}}T_{cos}}{b}=\lim_{b\rightarrow +\infty} \frac{\sup_{S^{T_{cmc}}_{b}}T_{cos}}{b}=1.$$
\end{cor}
\begin{proof}
Let $S^{T_{cmc}}_{b}$ be a $CMC$ level of constant mean curvature $-\frac{1}{b}$. We have  $k_{S^{T}_{b}}\geq-\frac{1}{b^{2}}$. Then by \cite[Remark~10.3]{BBZ2}, $S^{T_{cmc}}_{b}$ is in the future of the $k$-level $S^{T_{k}}_{b}$. We conclude using Proposition \ref{propmehdi3} and Remark \ref{remarmehdi3}. 
\end{proof} 

For every $a>0$, let $\Omega_{a}$ be the regular domain defined by $\Omega_{a}:=\frac{1}{a}\Omega$. Note that $\Omega_{a}$ is the regular domain associated to the cocycle $\frac{\tau}{a}$. The regular domain $\Omega_{a}$ converge when $a$ goes to $\infty$ to  the cone $\mathbf{C}$. Denote by $T^{a}_{cos}$, $T_{k}^{a}$ and $T_{cmc}^{a}$ respectively the cosmological time, the $k$-time and the $CMC$ time of $\Omega_{a}$. It is not hard to see that  $aT^{a}(x)=T^{1}(a x)$ for each of the three times. 
\begin{cor}
The Cauchy times $T_{k}^{a}$ (respectively $T_{cmc}^{a}$) converge in the compact open topology, when $a$ goes to $+\infty$, to the cosmological time of $\mathbf{C}$. That is for every compact $F$ of $\mathbf{C}$ and for $a$ big enough, the Cauchy time $T_{k}^{a}$(respectively $T_{cmc}^{a}$) converge uniformly on $F$ to the cosmological time of $\mathbf{C}$. 
\end{cor}

\begin{proof}
Let $F$ be a compact in $\mathbf{C}$. Note that for $a$ big enough $F\subset \Omega_{a}$. By \cite[Proposition~6.2]{bonsante1}, the cosmological time $T_{cos}^{a}$ converge uniformly on $F$ to the cosmological time of $\mathbf{C}$. So to proof that $T_{k}^{a}$ (respectively $T_{cmc}^{a}$) converge unifomly on $F$ to the cosmological time of $\mathbf{C}$, it is sufficent to proof that $\sup_{F}\left|T_{k}^{a}(x)-T^{a}_{cos}(x)\right|$ (respectively $\sup_{F}\left|T_{cmc}^{a}(x)-T^{a}_{cos}(x)\right|$) goes to $0$, when $a$ goes to $+\infty$.

$1)$ The $k$-time case. We have $$\sup_{F}\left|T_{k}^{a}(x)-T^{a}_{cos}(x)\right|\leq\left[1-\inf_{F}\frac{T^{1}_{cos}(ax)}{T_{k}^{1}(ax)}\right]\sup_{F}T_{k}^{a}(x)$$
Using Proposition \ref{propmehdi3}, one can see that $T_{k}^{a}(x)$ is bounded on $F$ and $\inf_{F}\frac{T^{1}_{cos}(ax)}{T^{1}_{k}(ax)}$ goes to $1$ when $a$ goes to $+\infty$. Thus we get that $\sup_{F}\left|T_{k}^{a}(x)-T^{a}_{cos}(x)\right|$ goes to $0$ when $a$ goes to $+\infty$. 

$2)$ The $CMC$-time case. We have $$\sup_{F}\left|T_{cmc}^{a}(x)-T^{a}_{cos}(x)\right|\leq\left[\sup_{F}\frac{T^{1}_{cos}(ax)}{T^{1}_{cmc}(ax)}-1\right]\sup_{K}T_{cmc}^{a}(x)$$
Then by  Corollary \ref{cormehdi3}, we have that $\sup_{F}\left|T_{cmc}^{a}(x)-T^{a}_{cos}(x)\right|$ goes to $0$ when $a$ goes to $+\infty$.

\end{proof}

\subsection{The de Sitter case}
Let $M\simeq B(S)/\Gamma$ be a $2+1$-dimensional $MGHC$ de Sitter space-time of hyperbolic type. Let $T_{k}$ be the $k$-time of $B(S)$. 
\begin{prop}
\label{propbelraoutidesitteralg}
We have:
\begin{itemize}
\item1) $\lim_{b\rightarrow +\infty} \frac{\inf_{S^{T_{k}}_{b}}T_{cos}}{\operatorname{argcoth}(\sqrt{\frac{b^{2}+1}{b^{2}}})}=\lim_{b\rightarrow +\infty} \frac{\sup_{S^{T_{k}}_{b}}T_{cos}}{\operatorname{argcoth}(\sqrt{\frac{b^{2}+1}{b^{2}}})}=1$;
\item2) There exists a constant $C>0$ such that $\lim_{b\rightarrow +\infty}\left[\sup_{S^{T_{k}}_{b}}T_{cos}-\inf_{S^{T_{k}}_{b}}T_{cos}\right]\leq C$.
\end{itemize}
\end{prop}
\begin{proof}
The proof is similar to the flat case. The $k$-level $S^{T_{k}}_{1}$ is of constant Gauss curvature $-2$. Let $H_{0}=\inf H_{S^{T_{k}}_{1}}$ and $H_{1}=\sup H_{S^{T_{k}}_{1}}$, where $H_{S^{T_{k}}_{1}}$ is the $CMC$ curvature of $S^{T_{k}}_{1}$.

Denote respectively by $T'_{cos}$, $r'$ the cosmological time and retraction map of the $\Gamma$-invariant future complete convex domain $A:=J^{+}(S^{T_{k}}_{1})$ of $B(S)$. For every $b>1$, let $p\in S^{T_{k}}_{b}$ such that $\inf_{S^{T_{k}}_{b}}T'_{cos}=T'_{cos}(p):=a$.

By the Maximum Principle we have $$k_{S^{T'_{cos}}_{a}}(p)\geq -\frac{1}{b^{2}}-1.$$
But $$k_{S^{T'_{cos}}_{a}}(p)=-\frac{2-2H_{S^{T_{k}}_{1}}(r'(p))\tanh(a)+\tanh^{2}(a)}{1-2H_{S^{T_{k}}_{1}}(r'(p))\tanh(a)+2\tanh^{2}(a)}$$ Thus $$\inf_{S^{T_{k}}_{b}}T_{cos}\geq \operatorname{argth}(\frac{H_{0}}{b^{2}+2}+\frac{1}{b^{2}+2}\sqrt{H_{1}^{2}+(b^{2}-1)(b^{2}+2)})$$

On the other hand and by the Maximum Principle we have $$\sup_{S^{T_{k}}_{b}}T_{cos}\leq \operatorname{argcoth}(\sqrt{\frac{b^{2}+1}{b^{2}}})$$

Hence $$1\geq\frac{\sup_{S^{T_{k}}_{b}}T_{cos}}{\operatorname{argcoth}(\sqrt{\frac{b^{2}+1}{b^{2}}})}\geq\frac{\inf_{S^{T_{k}}_{b}}T_{cos}}{\operatorname{argcoth}(\sqrt{\frac{b^{2}+1}{b^{2}}})}\geq \frac{\operatorname{argth}(\frac{H_{0}}{b^{2}+2}+\frac{1}{b^{2}+2}\sqrt{H_{1}^{2}+(b^{2}-1)(b^{2}+2)})}{\operatorname{argcoth}(\sqrt{\frac{b^{2}+1}{b^{2}}})}$$
Then a simple computation shows that:
\begin{itemize}
\item $\lim_{b\rightarrow +\infty} \frac{\inf_{S^{T_{k}}_{b}}T_{cos}}{\operatorname{argcoth}(\sqrt{\frac{b^{2}+1}{b^{2}}})}=\lim_{b\rightarrow +\infty} \frac{\sup_{S^{T_{k}}_{b}}T_{cos}}{\operatorname{argcoth}(\sqrt{\frac{b^{2}+1}{b^{2}}})}=1$
\item $\lim_{b\rightarrow +\infty}\left[\sup_{S^{T_{k}}_{b}}T_{cos}-\inf_{S^{T_{k}}_{b}}T_{cos}\right]\leq \frac{1}{2}\log(3-H_{0})$.
\end{itemize}

\end{proof}

\section{Bilipschitz control of convex hypersurfaces}
Let us consider $M$ to be a $n+1$-dimensional
\begin{itemize}
\item future complete flat standard $MGHC$ space-time;
\item or a future complete $MGHC$ de Sitter space-time of hyperbolic type;
\item or the tight past of a $MGHC$ anti de Sitter space-time.
\end{itemize}
Our next proposition shows that the geometry of a convex spacelike surface can be compared uniformly to the cosmological one. More precisely:  
\begin{prop}
\label{propositionmehdi3}
Let $S\subset \widetilde{M}$ be a $\pi_{1}(M)$-invariant convex Cauchy hypersurface of $\widetilde{M}$. Let $\mathbf{n}$ its  Gauss application and $N$ the cosmological normal application. Then for every  $p$ in $S$ we have,
\begin{itemize}
\item $\left|\left\langle N_{p},\mathbf{n}_{p}\right\rangle\right|\leq (\sup_{S}T_{cos})(\inf_{S}T_{cos})^{-1}$ if $M$ is flat;
\item $\left|\left\langle N_{p},\mathbf{n}_{p}\right\rangle\right|\leq (\sinh(\sup_{S}T_{cos}))(\sinh(\inf_{S}T_{cos}))^{-1}$ if $M$ is locally de Sitter;
\item $\left|\left\langle N_{p},\mathbf{n}_{p}\right\rangle\right|\leq (\tan(\sup_{S}T_{cos}))(\tan(\inf_{S}T_{cos}))^{-1}$ if $M$ is locally anti de Sitter.
\end{itemize}
\end{prop}
For the proof we need the following lemma:
\begin{lemm}
\label{lemmamehdi7}
Let $S^{T_{cos}}_{a}$ et $S^{T_{cos}}_{b}$ be two cosmological levels of $\widetilde{M}$ with $b<a$. Then for every  $p$ in  $S^{T_{cos}}_{b}$ and every unitary future oriented timelike tangent vector $x\in T_{p}\widetilde{M}$ such that $S^{T_{cos}}_{a}\subset J^{+}(P_{p})$, where $P_{p}=x^{\bot}\subset T_{p}\widetilde{M}$, we have: 
\begin{itemize}
\item $\left|\left\langle N_{p}, x\right\rangle\right|\leq (a)(b)^{-1}$ if $M$ is flat;
\item $\left|\left\langle N_{p}, x\right\rangle\right|\leq (\sinh(a))(\sinh(b))^{-1}$ if $M$ is locally de Sitter;
\item $\left|\left\langle N_{p}, x\right\rangle\right|\leq (\tan(a))(\tan(b))^{-1}$ if $M$ is locally anti de Sitter.
\end{itemize}
\end{lemm}
\textbf{Proof of Lemma \ref{lemmamehdi7} in the flat case}. Fix an origin of $\mathbb{R}^{1,n}$
and let $M\simeq \Omega/\Gamma_{\tau}$ be a flat standard future complete space-time. Let $p$ in $S^{T_{cos}}_{b}$ and let $x\in \mathbb{H}^{n}$ such that $S^{T_{cos}}_{a}\subset J^{+}(p+x^{\bot})$. For every $y$ in $S^{T_{cos}}_{a}$ we have: $$\left\langle y,x\right\rangle\leq \left\langle p,x\right\rangle$$ Then $$a\left\langle N_{y},x\right\rangle\leq b\left\langle N_{p},x\right\rangle+\left\langle r(p)-r(y),x\right\rangle.$$
The normal application $N:\Sigma^{T_{cos}}_{a}\rightarrow \mathbb{H}^{n}$ is surjective. So to conclude it is sufficient to take $y$ in $S^{T_{cos}}_{a}$ such that $N_{y}=x$. 

\begin{rmq}
\label{remarque1}
We restrict ourselves to standard space-times to get the surjectivity of the normal cosmological application. However, it still true in any future regular domain. Indeed, consider two cosmological levels  $S^{T_{cos}}_{a}$ and $S^{T_{cos}}_{b}$ with $b<a$. Let $p$ in $S^{T_{cos}}_{b}$ and let $S_{a}$ be the hyperboloid defined by:
$S_{a}=\left\{y\in J^{+}(r(p))\subset\mathbb{R}^{1,n}\mbox{~such that~} \left|y-r(p)\right|^{2}=-a^{2}\right\}.$
Remark that $S_{a}$ is in the future of $S^{T_{cos}}_{a}$. Thus for every  $y$ in $S_{a}$we have : $\left\langle y,x\right\rangle\leq \left\langle p,x\right\rangle$ and so $\left\langle y-r(p),x\right\rangle\leq b\left\langle N_{p},x\right\rangle$.
Then it is sufficient to take $y$ such that $y-r(p)=a x$
\end{rmq}\fin

\textbf{Proof of Lemma \ref{lemmamehdi7} in the de Sitter case}. Fix an origin of $\mathbb{R}^{1,n+1}$. Let $M\simeq B(S)/\Gamma$ be a $MGHC$ de Sitter space-time of hyperbolic type. Let $p$ in $S^{T_{cos}}_{b}$ and let $x$ be a unitary future oriented timelike tangent vector in $T_{p}\widetilde{M}$ such that $S^{T_{cos}}_{a}\subset J^{+}(P_{p})$, where $P_{p}=x^{\bot}\subset T_{p}B(S)$. The proof is similar to the one of  Remark \ref{remarque1} which depends only on $J^{+}(r(p))$.  Note that $J^{+}(r(p))$ is isometric to a domain of $\mathbb{D}\mathbb{S}_{n+1}$. So we can, without losing generality, restric ourselves and work in $dS_{n+1}$. 

For every $y$ in the hypersurface $S_{a}=\left\{y\in J^{+}(r(p))\subset dS_{n+1} \mbox{~such that~} d_{Lor}(y,r(p))=a\right\}$. We have,  $$\left\langle x, p-y\right\rangle\geq 0,$$ where $\left\langle . , .\right\rangle$ is the scalar product of $\mathbb{R}^{1,n+1}$. Thus $$0=\left\langle x,p\right\rangle\geq \left\langle x,y\right\rangle$$
Let us write:
\begin{itemize}
\item $r(p)=-\left\langle r(p), x \right\rangle x+ u'$, where $u'\in x^{\bot}$;
\item $p=\cosh(b) r(p)+\sinh(b) N_{r(p)}$, where $N_{r(p)}\in \mathbb{H}^{n+1}\cap T_{r(p)}dS_{n+1}$ is the cosmological normal vector;
\item $y=\cosh(a) r(p)+\sinh(a) v_{y}$, where $v_{y}\in \mathbb{H}^{n+1}\cap T_{r(p)}dS_{n+1}$.
\end{itemize}

Now take $v_{y}=(\sqrt{1+\left\langle r(p), x \right\rangle^{2}}) x-\frac{\left\langle r(p),x\right\rangle}{\left|u'\right|}u'$. 

On the one hand $\left\langle x,p\right\rangle=0$ so, $$\left\langle x,N_{p}\right\rangle=-\frac{1}{\sinh(b)}\left\langle x, r(p)\right\rangle$$
On the other hand $\left\langle x,y\right\rangle\leq0$ and hence, $$\left\langle x,r(p)\right\rangle\leq \sinh(a)$$
Thus $$\left|\left\langle x,N_{p}\right\rangle\right|\leq\frac{\sinh(a)}{\sinh(b)}$$\\

\textbf{Proof of  Lemma \ref{lemmamehdi7} in the anti de Sitter case}. Fix an origin of $\mathbb{R}^{2,n}$.
Let $M$ be the tight past of a $MGHC$ anti de Sitter space-time. Note that $\widetilde{M}$ is isometric to a domain of $\mathbb{A}\mathbb{D}\mathbb{S}_{n+1}$. Let $p$ in $S^{T_{cos}}_{b}$ and let $x\in AdS_{n+1}\subset\mathbb{R}^{2,n}$ such that $P_{p}=x^{\bot}\subset T_{p}\widetilde{M}$, $x$ is future oriented (with respect to the orientation of  $ADS_{n+1}$) and $S^{T_{cos}}_{a}\subset J^{+}(P_{p})$. Let $S_{a}=\left\{y\in J^{+}(r(p))\subset AdS_{n+1} \mbox{~such that~} d_{Lor}(y,r(p))=a\right\}$. For every $y$ in $S_{a}$ we have, $$\left\langle x, p-y\right\rangle\geq 0,$$ where $\left\langle . , .\right\rangle$ is the scalar product of $\mathbb{R}^{2,n}$. Thus $$0=\left\langle x,p\right\rangle\geq \left\langle x,y\right\rangle$$
Let us write:
\begin{itemize}
\item $r(p)=-\left\langle r(p), x \right\rangle x-\left\langle p, r(p) \right\rangle p+ u'$, where $u'\in Vect(x,p)^{\bot}$;
\item $p=\cos(b) r(p)+\sin(b) N_{r(p)}$, where $N_{r(p)}\in AdS_{n+1}\cap T_{r(p)}AdS_{n+1}$ is the cosmological normal vector;
\item $y=\cos(a) r(p)+\sin(a) v_{y}$, where $v_{y}\in AdS_{n+1}\cap T_{r(p)}AdS_{n+1}$ is future oriented.
\end{itemize}
We get then:
\begin{itemize}
\item $\left\langle x, N_{p}\right\rangle=-\frac{1}{\sin(b)}\left\langle x, r(p)\right\rangle$;
\item $\left\langle x, r(p)\right\rangle\leq -\tan(a)\left\langle x,v_{y}\right\rangle$;
\end{itemize}

For every $\beta\in \mathbb{R}$ let,  $$v(\beta)=(-\left\langle p,r(p)  \right\rangle) x+\beta p+\frac{\sqrt{\left\langle p,r(p)  \right\rangle^{2}+\beta^{2}-1}}{\left|u'\right|}u'.$$
Note that there exists  $\beta$ such that: $\left|v(\beta)\right|^{2}=-1$ and $\left\langle v(\beta),r(q)\right\rangle=0$. In this case $v(\beta)$ is future oriented. Indeed, if not then $v_{y}=-v(\beta)$ is future oriented and hence $y=\cos(a) r(p)+\sin(a) v_{y}$ belongs to $S_{a}$. But $\left\langle r(p), x \right\rangle\leq \tan(a)\left\langle v(\beta), x \right\rangle=-\cos(b)\tan(a)\leq 0$ and $\left\langle r(p), x \right\rangle=-\sin(b)\left\langle N_{p}, x \right\rangle\geq 0$  which is a contradiction.
Thus $y=\cos(a) r(p)+\sin(a) v_{\beta}$ belongs to $S_{a}$ and hence $$\left|\left\langle N_{p}, x\right\rangle\right|\leq (\tan(a))(\tan(b))^{-1}$$.

\textbf{Proof of Proposition \ref{propositionmehdi3}}
Denote by $a=\sup_{S}T_{cos}$ and $b=\inf_{S}T_{cos}$. The hypersurface $S$ is in the past of $S^{T_{cos}}_{a}$ and in the future of $S^{T_{cos}}_{b}$. Let $p$ in $S$ and let $P_{p}=\mathbf{n}_{p}^{\bot}$ the tangent hyperplane to $S$ at $p$. As $S$ is convex, we have that $S^{T_{cos}}_{a}\subset J^{+}(P_{p})$. 
By Lemma \ref{lemmamehdi7} we have:
\begin{itemize} 
\item $\left|\left\langle N_{p},\mathbf{n}_{p}\right\rangle\right|\leq \frac{a}{T_{cos}(p)}\leq \frac{a}{b}$ in the flat case;
\item $\left|\left\langle N_{p},\mathbf{n}_{p}\right\rangle\right|\leq \frac{\sinh(a)}{\sinh(T_{cos}(p))}\leq \frac{\sinh(a)}{\sinh(b)}$ in the de Sitter case;
\item $\left|\left\langle N_{p},\mathbf{n}_{p}\right\rangle\right|\leq \frac{\sin(a)}{\sin(T_{cos}(p))}\leq \frac{\tan(a)}{\tan(b)}$ in the anti de Sitter case.
\end{itemize}
and this concludes the proof.\fin
\subsection{The (n+1)-flat case}
Let $M\simeq\Omega/\Gamma$ be a future complete $MGHC$ flat non elementary  space-time of dimension $n+1$.
\begin{prop}
\label{prop.mehdi456677}
Let $S\subset \Omega$ be a convex $\Gamma$ invariant Cauchy  hypersurface and let $g_{S}$ be the Riemannian metric defined on $S$ by the restriction of the ambient Lorentzian metric of  the Minkowski space $\mathbb{R}^{1,n}$. Then  $(S,g_{S})$ is $K^{4}$-bi-Lipschitz to $(S^{T_{cos}}_{\sup_{S}T_{cos}}, g^{T_{cos}}_{\sup_{S}T_{cos}})$, where $K=\frac{\sup_{S}T_{cos}}{\inf_{S}T_{cos}}$.
\end{prop}

\begin{rmq}
The fact that $(S,g_{S})$ is bi-Lipschitz to $(\Sigma^{T_{cos}}_{\sup_{S}T_{cos}}, g^{T_{cos}}_{\sup_{S}T_{cos}})$ is a direct consequence of the cocompactness of the $\Gamma$-action. What we are proving here is that the bi-Lipschitz constant  $K$ depend only on the cosmological barrier and not on the hypersurface $S$.
\end{rmq}

Let us start with the following proposition due to Bonsante:
\begin{prop}{\em (\cite[Lemme~7.4]{bonsante1})}.
\label{bonsanteporputilisépreuve}
The cosmological levels $S^{T_{cos}}_{a}$ and $S^{T_{cos}}_{b}$ are $(\frac{a}{b})^{2}$-bi-Lipschitz one to the other. More precisely, $$g_{b}\leq g_{a}\leq (\frac{a}{b})^{2} g_{b}$$

\end{prop}

\textbf{Proof of Proposition \ref{prop.mehdi456677}}.
Let $S$ be a  convex $\Gamma$ invariant Cauchy hypersurface of $\Omega$ and let $g_{S}$ its induced Riemannian metric. Denote by $a=\sup_{S}T_{cos}$ and by $b=\inf_{S}T_{cos}$. 


Let $\alpha:[0,1]\rightarrow S$ be a Lipschitz curve in $S$. For almost every $s$ in $[0,1]$, we have $$\dot{\alpha}(s)=\dot{r}(s)+\dot{T}_{cos}(s)N_{s}+T_{cos}(s)\dot{N}(s)$$ And hence $$\left|\dot{\alpha}(s)\right|^{2}=\left|\dot{r}(s)+T_{cos}(s)\dot{N}(s)\right|^{2}-\dot{T}_{cos}(s)^{2}$$
Thus by Proposition \ref{bonsanteporputilisépreuve} $$\left|\dot{\alpha}(s)\right|^{2}\geq \left|\dot{r}(s)+b\dot{N}(s)\right|^{2}-\dot{T}_{cos}(s)^{2}.$$

Note that $$\dot{T}_{cos}(s)=d_{\alpha(s)}T_{cos}.\dot{\alpha}(s)=\left\langle N(\alpha(s)),\dot{\alpha}(s)\right\rangle.$$
Let us write $N(\alpha(s))=h(s)\mathbf{n}_{\alpha(s)}+v(s)$, where $\mathbf{n}$ is the normal map of $S$ and $v(s)$ is in $\mathbf{n}(\alpha(s))^{\bot}$. By Proposition \ref{propositionmehdi3}: $\left|\left\langle N_{\alpha(s)},\mathbf{n}(\alpha(s))\right\rangle\right|\leq  \frac{a}{b}$ and hence $\left|v(s)\right|^{2}\leq (\frac{a}{b})^{2}-1$.

But $$\left|\dot{T}_{cos}(s)\right|=\left|\left\langle v(s),\dot{\alpha}(s)\right\rangle\right|\leq \left|v(s)\right|\left|\dot{\alpha}(s)\right|$$ Thus $$\dot{T}_{cos}(s)^{2}\leq ((\frac{a}{b})^{2}-1)\left|\dot{\alpha}(s)\right|^{2}.$$

Which proves that $$(\frac{b}{a})^{2}\left|\dot{r}(s)+b\dot{N}(s)\right|^{2}\leq\left|\dot{\alpha}(s)\right|^{2}$$ 

On the other hand and by Proposition \ref{bonsanteporputilisépreuve} we have $$\left|\dot{\alpha}(s)\right|^{2}\leq  \left|\dot{r}(s)+a\dot{N}(s)\right|^{2}\leq(\frac{a}{b})^{2}\left|\dot{r}(s)+b\dot{N}(s)\right|^{2}$$ 
Thus $$(\frac{b}{a})^{4}\left|\dot{r}(s)+a\dot{N}(s)\right|^{2}\leq\left|\dot{\alpha}(s)\right|^{2}\leq \left|\dot{r}(s)+a\dot{N}(s)\right|^{2}$$
This proves that the cosmological flow induces a $(\frac{b}{a})^{4}$-bi-Lipschitz identification between $(S,g_{S})$ and $(S^{T_{cos}}_{a}, g^{T_{cos}}_{a})$.
\fin

\begin{cor}
Let $M$ be a $MGHC$ flat future complete non elementary space-time. Let $T_{cmc}:\widetilde{M}\rightarrow \mathbb{R}_{+}$ its associated Cauchy time. Then for every $a>0$, the hypersurface $(S^{T_{cmc}}_{a}, g^{T_{cmc}}_{a})$ is $n^{4}$-bi-Lipschitz to the hypersurface $(S^{T_{cos}}_{a},g^{T_{cos}}_{a})$.
\end{cor}

\begin{proof}
The corollary follows from Remark \ref{remarmehdi3} and Proposition \ref{prop.mehdi456677}.
\end{proof}
\subsection{The (2+1)-de Sitter case}
Suppose now that $M\simeq B(S)/\Gamma$ is a $MGHC$ de Sitter space-time of hyperbolic type and of dimension $2+1$. Let $(\Omega_{1},\mathfrak{g})$ be the hyperbolic $dS$-standard space-time of dimension $2+1$ associated to $M$ obtained by a Wick rotation from a flat regular domain $(\Omega_{1},g)$. Let $T_{cos}$ and $\mathcal{T}_{cos}$ be respectively the cosmological time of $(\Omega,g)$ and $(\Omega_{1},\mathfrak{g})$.
\begin{prop}
\label{promehdi49}
The cosmological levels $S^{\mathcal{T}_{cos}}_{a}$ and $S^{\mathcal{T}_{cos}}_{b}$ of $B(S)$ are $(\frac{\sinh(a)}{\sinh(b)})^{2}$-bi-Lipschitz one to the other. More precisely,
$$\mathfrak{g}^{\mathcal{T}_{cos}}_{b}\leq \mathfrak{g}^{\mathcal{T}_{cos}}_{a}\leq (\frac{\sinh(a)}{\sinh(b)})^{2}\mathfrak{g}^{\mathcal{T}_{cos}}_{b}$$
\end{prop}
\begin{proof}
Suppose that $b<a$. We have $$\mathfrak{g}^{\mathcal{T}_{cos}}_{a}=\frac{1}{(1-\tanh^{2}(a))}g^{T_{cos}}_{\tanh(a)}$$ 
But by Proposition \ref{bonsanteporputilisépreuve} $$g^{T_{cos}}_{\tanh(b)}\leq g^{T_{cos}}_{\tanh(a)}\leq (\frac{\tanh(a)}{\tanh(b)})^{2}g^{T_{cos}}_{\tanh(b)}$$ Thus $$\mathfrak{g}^{\mathcal{T}_{cos}}_{b}\leq \mathfrak{g}^{\mathcal{T}_{cos}}_{a}\leq (\frac{\sinh(a)}{\sinh(b)})^{2}\mathfrak{g}^{\mathcal{T}_{cos}}_{b}.$$
\end{proof}
\begin{prop}
\label{mehdipropantidesitteralg}
Let $S\subset B(S)$ be a convex $\Gamma$ invariant Cauchy  hypersurface and let $\mathfrak{g}_{S}$ be the metric of $S$. Then,
$(S,\mathfrak{g}_{S})$ is $K^{4}$-bi-Lipschitz to $(S^{\mathcal{T}_{cos}}_{\sup_{S}\mathcal{T}_{cos}}, \mathfrak{g}^{\mathcal{T}_{cos}}_{\sup_{S}\mathcal{T}_{cos}})$, where\\
$K=\frac{\sinh(\sup_{S}\mathcal{T}_{cos})}{\sinh(\inf_{S}\mathcal{T}_{cos})}$.

\end{prop}
\begin{proof}
Let us denote for simplicity by $a=\sup_{S}\mathcal{T}_{cos}$, by $b=\inf_{S}\mathcal{T}_{cos}$ and by $\left|.\right|_{1}$ the de Sitter norm of $\Omega_{1}$.  Let $\alpha:[0,1]\rightarrow S$ be a Lipschitz curve in $\Omega_{1}$. For almost every $s$ in $[0,1]$ we have, $$\left|\dot{\alpha}(s)\right|_{1}^{2}=\frac{1}{1-T_{cos}^{2}(s)}\left|\dot{r}(s)+T_{cos}(s)\dot{N}(s)\right|^{2}-\frac{1}{(1-T_{cos}^{2}(s))^{2}}\dot{T}_{cos}^{2}(s)$$ Thus by Proposition \ref{promehdi49}, $$\left|\dot{r}(s)+\tanh(b)\dot{N}(s)\right|^{2}_{1}-\dot{\mathcal{T}}_{cos}^{2}(s)\leq \left|\dot{\alpha}(s)\right|_{1}^{2}\leq \left|\dot{r}(s)+\tanh(a)\dot{N}(s)\right|^{2}_{1}$$
Using the same arguments as in the flat case we get that, $$\dot{\mathcal{T}}_{cos}^{2}(s)\leq ((\frac{\sinh(a)}{\sinh(b)})^{2}-1)\left|\dot{\alpha}(s)\right|_{1}^{2}$$
Hence $$\left|\dot{\alpha}(s)\right|_{1}^{2}\geq (\frac{\sinh(b)}{\sinh(a)})^{2} \left|\dot{r}(s)+\tanh(b)\dot{N}(s)\right|^{2}_{1}$$
Then by Proposition \ref{promehdi49} we get $$(\frac{\sinh(b)}{\sinh(a)})^{4}\left|\dot{r}(s)+\tanh(a)\dot{N}(s)\right|^{2}_{1}\leq\left|\dot{\alpha}(s)\right|_{1}^{2}\leq  \left|\dot{r}(s)+\tanh(a)\dot{N}(s)\right|^{2}_{1}.$$ 
\end{proof}
\begin{rmq}
\label{jeudivac1}
Actually in Proposition \ref{mehdipropantidesitteralg} we proved that if $\alpha$ is a spacelike curve contained in the past of the cosmological level $S^{T_{cos}}_{a}$, then the length $l(\alpha)$ of $\alpha$ is less than the length of $\Phi_{T_{cos}}^{a-T_{cos}(\alpha)}$, where $\Phi_{T_{cos}}$ is the cosmological flow.
\end{rmq}

\subsection{The (2+1)-anti de Sitter case}
Let $M$ be the tight past of a $MGHC$ anti de Sitter space-time of dimension $2+1$. Recall that  $\widetilde{M}\simeq (\Omega,\mathfrak{g})$, where $(\Omega,\mathfrak{g})$ is obtained by a Wick rotation from a flat regular domain $(\Omega,g)$. Let $T_{cos}$ and $\mathcal{T}_{cos}$ be respectively the cosmological time of $(\Omega,g)$ and $(\Omega,\mathfrak{g})$.

\begin{prop}
\label{promehdi50}
The cosmological levels $S^{\mathcal{T}_{cos}}_{a}$ and $S^{\mathcal{T}_{cos}}_{b}$ of  $\widetilde{M}$ are $(\frac{\tan(a)}{\tan(b)})^{2}$-bi-Lipschitz one to the other. More precisely, $$(\frac{\cos(a)}{\cos(b)})^{2}\mathfrak{g}^{\mathcal{T}_{cos}}_{b}\leq \mathfrak{g}^{\mathcal{T}_{cos}}_{a}\leq (\frac{\sin(a)}{\sin(b)})^{2}\mathfrak{g}^{\mathcal{T}_{cos}}_{b}$$
\end{prop}
\begin{proof}
Suppose that $b<a$. We have  $$\mathfrak{g}^{\mathcal{T}_{cos}}_{a}=\frac{1}{1+\tan^{2}(a)}g^{T_{cos}}_{\tanh(a)}$$ 
But by  Proposition \ref{bonsanteporputilisépreuve} $$g^{T_{cos}}_{\tan(b)}\leq g^{T_{cos}}_{\tan(a)}\leq (\frac{\tan(a)}{\tan(b)})^{2}g^{T_{cos}}_{\tan(b)}$$ Thus $$(\frac{\cos(a)}{\cos(b)})^{2}\mathfrak{g}^{\mathcal{T}_{cos}}_{b}\leq \mathfrak{g}^{\mathcal{T}_{cos}}_{a}\leq (\frac{\sin(a)}{\sin(b)})^{2}\mathfrak{g}^{\mathcal{T}_{cos}}_{b}.$$
\end{proof}

\begin{prop}
\label{prop123546}
Let $S\subset \widetilde{M}$ be a convex $\Gamma$ invariant Cauchy  surface and let $\mathfrak{g}_{S}$ be the metric of $S$. Then $(S,\mathfrak{g}_{S})$ is $K^{4}$-bi-Lipschitz to $(S^{\mathcal{T}_{cos}}_{\sup_{S}\mathcal{T}_{cos}}, \mathfrak{g}^{\mathcal{T}_{cos}}_{\sup_{S}\mathcal{T}_{cos}})$, where\\
$K=\frac{\tan(\sup_{S}\mathcal{T}_{cos})}{\tan(\inf_{S}\mathcal{T}_{cos})}$
\end{prop}
\begin{proof}
Let us denote for simplicity by $a=\sup_{S}\mathcal{T}_{cos}$, by $b=\inf_{S}\mathcal{T}_{cos}$ and by $\left|.\right|_{1}$ the anti de Sitter norm of $\Omega$. Let $\alpha:[0,1]\rightarrow S$ be a Lipschitz curve in $S$. For almost every $s$ in $[0,1]$ we have, $$\left|\dot{\alpha}(s)\right|_{1}^{2}=\frac{1}{1+T_{cos}^{2}(s)}\left|\dot{r}(s)+T_{cos}(s)\dot{N}(s)\right|^{2}-\frac{1}{(1+T_{cos}^{2}(s))^{2}}\dot{T}_{cos}^{2}(s)$$
Thus by Proposition \ref{promehdi50}, $$(\frac{\cos(a)}{\cos(b)})^{2}\left|\dot{r}(s)+\tan(b)\dot{N}(s)\right|^{2}_{1}-\dot{\mathcal{T}}_{cos}^{2}(s)\leq \left|\dot{\alpha}(s)\right|_{1}^{2}\leq (\frac{\cos(b)}{\cos(a)})^{2}\left|\dot{r}(s)+\tan(a)\dot{N}(s)\right|^{2}_{1}$$
Using the same arguments as in the flat and the de Sitter case we get that, $$\dot{\mathcal{T}}_{cos}^{2}(s)\leq ((\frac{\tan(a)}{\tan(b)})^{2}-1)\left|\dot{\alpha}(s)\right|_{1}^{2}$$
Hence $$\left|\dot{\alpha}(s)\right|_{1}^{2}\geq (\frac{\sin(b)}{\sin(a)})^{2} \left|\dot{r}(s)+\tan(b)\dot{N}(s)\right|^{2}_{1}$$
Then by Proposition \ref{promehdi50} we get $$(\frac{\tan(b)}{\tan(a)})^{4}\left|\dot{r}(s)+\tan(a)\dot{N}(s)\right|^{2}_{1}\leq\left|\dot{\alpha}(s)\right|_{1}^{2}\leq (\frac{\tan(a)}{\tan(b)})^{4}\left|\dot{r}(s)+\tan(a)\dot{N}(s)\right|^{2}_{1}.$$ 
\end{proof}
\begin{rmq}
\label{jeudivac2}
Actually in Proposition \ref{prop123546} we proved that if $\alpha$ is a spacelike curve contained in the past of the cosmological level $S^{T_{cos}}_{a}$ and in the future of the cosmological level $S^{T_{cos}}_{b}$, then the length $l(\alpha)$ of $\alpha$ is less than  $\frac{\cos(b)}{\cos(a)}l(\Phi_{T_{cos}}^{a-T_{cos}(\alpha)})$, where $\Phi_{T_{cos}}$ is the cosmological flow.
\end{rmq}

\section{Asymptotic behavior in flat (n+1)-space-times}

\subsection{Geometric properties of the  initial singularity}
In this part we will prove the three first points of  Theorem \ref{mehditheorem2}. Let $\Omega$ be a flat future complete regular domain and let $(\Sigma,d_{\Sigma})$, $(\partial \Omega/\sim, \bar{d}_{\partial \Omega})$ be the Initial Singularity and the Horizon associated to $\Omega$. Denote by $(\Sigma^{\bigstar}, d^{\bigstar}_{\Sigma})$ the completion of $(\Sigma,d_{\Sigma})$.
By a result of Bonsante \cite{bonsante1} the metric space $(\Sigma,d_{\Sigma})$ embed isometrically in $(\partial \Omega/\sim, \bar{d}_{\partial \Omega})$.
\begin{prop}
The Horizon $(\partial \Omega/\sim, \bar{d}_{\partial \Omega})$ embeds isometrically in $(\Sigma^{\bigstar}, d^{\bigstar}_{\Sigma})$.
\end{prop}
\begin{proof}
Let $\Sigma_{\infty}$ be the set of Cauchy sequences of $(\Sigma,d_{\Sigma})$ and let $d_{\infty}$ the pseudo-distance defined by: if $(x_{i})_{i\in \mathbb{N}}$ and $(y_{i})_{i\in \mathbb{N}}$ are two Cauchy sequences of $(\Sigma,d_{\Sigma})$, then $d_{\infty}((x_{i})_{i},(y_{i})_{i})=\lim_{i\rightarrow \infty} d_{\Sigma}(x_{i},y_{i})$. Denote by $\pi':\Sigma_{\infty}\rightarrow \Sigma^{\bigstar}$ the projection of $\Sigma_{\infty}$ in $\Sigma^{\bigstar}$. 

Let $x$ in $\partial \Omega\setminus \Sigma$ and let $(p_{i})_{i\in\mathbb{N}}$ be a sequence of $\Omega$ converging to $x$ and such that $T_{cos}(p_{i+1})<T_{cos}(p_{i})$,  for every $i$ in $\mathbb{N}$. Note that the sequence $r(p_{i})$ stays in a compact of $\partial \Omega$. Thus extract a subsequence if necessary, we can suppose that  $r(p_{i})$ converges to $y$ in $\partial \Omega$. Then the vector $x-y$ is a causal vector. But $\partial \Omega$ is achronal so $x-y$ is lightlike. Hence $y$ should belongs to the lightlike ray passing through $x$ which is unique by Lemma \cite[Lemma~4.11]{bonsante1}. Thus for every $x$ in $\partial \Omega$, there exists a sequence $(x_{i})_{i\in \mathbb{N}}$ of $\Sigma$ converging to a point $y\in\overline{\Sigma}$ such that $d_{\partial \Omega}(x,y)=0$. 

Now let $f:\partial \Omega\rightarrow \Sigma^{\bigstar}$ be the function which associates  to each $x$ in $\partial \Omega$ the image by $\pi'$ of a sequence $(x_{i})_{i\in \mathbb{N}}$ in $\Sigma$ converging to a point $y$ of $\overline{\Sigma}$ such that $d_{\partial \Omega}(x,y)=0$. This function is well defined and induces an isometric embedding from $(\partial \Omega/\sim, \bar{d}_{\partial \Omega})$ to $(\Sigma^{\bigstar}, d^{\bigstar}_{\Sigma})$.
\end{proof}

\begin{prop}
\label{propmercredi}
For every $x$ and $y$ in $\Sigma$, there exists a geodesic in $(\partial \Omega^{+}/\sim, \bar{d}_{\partial \Omega^{+}})$ joining $x$ and $y$.
\end{prop}
We will need the following lemma:
\begin{lemm}
\label{lemmmehdimb}
Consider the Lorentzian model $\mathbb{H}^{n}$ of the hyperbolic space. For every $n_{1}\neq n_{2}$ in $\mathbb{H}^{n}$, the subset defined by $F=\left\{v\in dS_{n} \mbox{~such that~} \left\langle v,n_{1}\right\rangle\geq 0 \mbox{~and~} \left\langle v,n_{2}\right\rangle\leq 0\right\}$ is precompact.
\end{lemm}
\begin{proof}
Fix an origin of the Minkowski space $\mathbb{R}^{1,n}$.
Let $n_{1}$ and $n_{2}$ in $\mathbb{H}^{n}$ and $v$ in $\mathbb{R}^{1,n}$ such that $\left|v\right|^{2}=1$, $\left\langle v,n_{1}\right\rangle\geq 0$ and $\left\langle v,n_{2}\right\rangle\leq 0$.

 One can write $v=-\left\langle v,n_{1}\right\rangle n_{1}+v_{1}$, where $v_{1}$ is in $n_{1}^{\bot}$. We have then
$$-\left\langle v,n_{1}\right\rangle^{2}+\left|v_{1}\right|^{2}=1,$$
And hence $$\left|\left|v\right|\right|^{2}=1+2\left\langle v,n_{1}\right\rangle^{2},$$ where $\left|\left|.\right|\right|$ is the euclidean norm of $\mathbb{R}^{n+1}$.

Thus if we want to proof that  $v$ stays in a compact, we need to proof that $\left\langle v,n_{1}\right\rangle$ is bounded independently of $v$.

In the same way we can write $n_{2}=-\left\langle n_{1},n_{2} \right\rangle n_{1}+ u_{1}$, where $u_{1}$ is in $n_{1}^{\bot}$. Thus $$-(\left\langle n_{2},n_{1}\right\rangle)^{2}+\left|u_{1}\right|^{2}=-1.$$
But $\left\langle v,n_{2}\right\rangle\leq 0$, so $$-\left\langle n_{1},n_{2}\right\rangle\left\langle n_{1},v\right\rangle+\left\langle v_{1},u_{1}\right\rangle\leq 0,$$ Hence $$0\leq\left\langle v,n_{1}\right\rangle\leq \frac{\left\langle v_{1},u_{1}\right\rangle}{\left\langle n_{1},n_{2}\right\rangle}.$$

Let's write $v_{1}=-\left\langle v_{1},u_{1}\right\rangle  u_{1}+v'_{1}$, where $v'_{1}$ is in $n_{1}^{\bot}\cap u_{1}^{\bot}$. Thus $$-\left\langle v,n_{1}\right\rangle^{2}+\left\langle v_{1},u_{1}\right\rangle^{2}+\left|v'_{1}\right|^{2}=1,$$ Then $$\left\langle v_{1},u_{1}\right\rangle^{2}\leq \frac{\left\langle n_{1},n_{2}\right\rangle^{2}}{\left\langle n_{1},n_{2}\right\rangle^{2}-1}$$ And this proves that $$0\leq\left\langle v,n_{1}\right\rangle\leq \frac{1}{\sqrt{\left\langle n_{1},n_{2}\right\rangle^{2}-1}}.$$
\end{proof}
\begin{prop}
\label{provrt}
Let $\alpha:[0,l]\rightarrow S^{T_{cos}}_{a} $ be the geodesic joining two point $p$ and $q$ of $S^{T_{cos}}_{a}$. Then for every $s$ in $[0,l]$ we have $$\left\langle \dot{\alpha}(s), N_{p}\right\rangle \leq 0 \mbox{~and~} \left\langle \dot{\alpha}(s), N_{q}\right\rangle \geq 0,$$ where $N_{p}$ and $N_{q}$ are the normal vectors of $S^{T_{cos}}_{a}$ at $p$ and $q$ respectively.
\end{prop}
\begin{proof}
Let $(x_{0},x_{1},...,x_{n})$ be a coordinate system of $\mathbb{R}^{1,n}$ such that $p=(0,...,0)$ and $N_{p}=(1,0,...,0)$. The hypersurface $S^{T_{cos}}_{a}$ is the graph of $1$-Lipschitz convex $C^{1}$ function $\phi:\mathbb{R}^{n}\rightarrow \mathbb{R}$. We have $\left\langle \dot{\alpha}(s),N_{p}\right\rangle=-\dot{\phi}(s)$. By \cite[Lemma~7.7]{bonsante1}, $\phi$ is increasing, hence $\left\langle \dot{\alpha}(s),N_{p}\right\rangle\leq 0$. In the same way we prove that $\left\langle \dot{\alpha}(s),N_{p}\right\rangle\geq 0$.

\end{proof}

Let $T_{cos}$ the cosmological time of $\Omega$ and consider $X_{T_{cos}}$ the space of  gradient lines of $T_{cos}$. Note that the normal application and the retraction map of $\Omega$ can be seen as maps on $X_{T_{cos}}$.

\textbf{Proof of Proposition \ref{propmercredi}}.
Let $\pi:\partial \Omega\rightarrow \partial \Omega/\sim $ be the projection of $\partial \Omega$ in $\partial \Omega/\sim $. Note that if $F$ is a compact of $\partial \Omega\subset \mathbb{R}^{1,n}$, then $\pi(F)$ is a compact of $(\Omega/\sim,\bar{d}_{\partial \Omega})$. Let $d_{euc}$ be the euclidean metric of $\mathbb{R}^{n+1}$ and $L_{euc}$ its associated euclidean length structure. Denote by $L$ the length structure defined on $\partial \Omega/\sim$ by the distance $\bar{d}_{\partial \Omega^{+}}$ and by $\mathcal{L}$ the one induced by the Minkowski metric.
 
 We want to proof that for every $\mathbf{p}$ and $\mathbf{q}$ in $X_{T_{cos}}$, there is a geodesic in $(\partial \Omega^{+}/\sim, \bar{d}_{\partial \Omega^{+}})$ joining $r(\mathbf{p})$ and $r(\mathbf{q})$. There are two distinct cases:

$1$) If $N_{\mathbf{p}}=N_{\mathbf{q}}$. Then by Proposition \cite[Proposition~4.14]{bonsante1},  $r(\mathbf{p})+ s(r(\mathbf{p})-r(\mathbf{q}))$ is contained in $\partial \Omega$ for every $s$ in $[0,1]$. Clearly $r(\mathbf{p})+ s(r(\mathbf{p})-r(\mathbf{q}))$ is a geodesic in $(\partial \Omega/\sim, \bar{d}_{\partial \Omega})$ joining $r(\mathbf{p})$ and $r(\mathbf{q})$.

$2$) If $N_{\mathbf{p}}\neq N_{\mathbf{q}}$. For every $0<a<1$, let $\alpha_{a}:[0,l_{a}]\rightarrow S^{T_{cos}}_{a}$ be the geodesic joining $\mathbf{p}$ and $\mathbf{q}$ i.e joining the intersection point of $\mathbf{p}$ and $ S^{T_{cos}}_{a}$ with the intersection point of $\mathbf{q}$ and $S^{T_{cos}}_{a}$. By Proposition \ref{provrt} we have $\left\langle \dot{\alpha}_{a}(s), N_{\mathbf{p}}\right\rangle \leq 0 \mbox{~and~} \left\langle \dot{\alpha}_{a}(s), N_{\mathbf{q}}\right\rangle \geq 0$, for every $s$ in $[0,l_{a}]$.
Therefore, by Lemma \ref{lemmmehdimb}, there is a compact $F\subset dS_{n}\subset \mathbb{R}^{n+1}$ such that $\dot{\alpha}_{a}(s)\in F$ for every $0<a<1$ and every $s$ in $[0,l_{a}]$. There is hence a constant $C>0$ such that $L_{euc}(\alpha_{a})\leq C$ for every $0<a<1$. This means that the geodesics $\alpha_{a}$ are contained in a compact $F'$ of $\overline{\Omega}$.

On the one hand, as $J^{-}(F')\cap \partial \Omega$ is compact in $\partial \Omega$, the curves $\pi\circ r\circ\alpha_{a}$ stay in a compact  of  $(\partial \Omega/\sim, \bar{d}_{\partial \Omega})$.

On the other hand, for every $0<a<1$ and every $s_{1}$, $s_{2}$ in $[0,l_{a}]$ we have,
$$\bar{d}_{\partial \Omega}(\pi(r(\alpha_{a}(s_{1}))),\pi(r(\alpha_{a}(s_{2})))) = d_{\Sigma}(r(\alpha_{a}(s_{1})),r(\alpha_{a}(s_{2}))).$$ But by \cite[Lemma~7.4, Proposition~7.8]{bonsante1},
 $$d_{\Sigma}(r(\alpha_{a}(s_{1})),r(\alpha_{a}(s_{2})))\leq  d^{T_{cos}}_{a}(\alpha_{a}(s_{1})), \alpha_{a}(s_{2})))$$

And hence, $$\bar{d}_{\partial \Omega}(\pi(r(\alpha_{a}(s_{1}))),\pi(r(\alpha_{a}(s_{2}))))\leq \left|s_{1}-s_{2}\right|.$$
This proves that the family $(\pi\circ r\circ\alpha_{a})_{0<a<1}$ is an equicontinuous family of curves. Thus by the Ascoli-Arzéla Theorem we deduce that $\pi\circ r\circ\alpha_{a}$ converges uniformly in $(\partial \Omega/\sim, \bar{d}_{\partial \Omega})$ to a curve $\alpha$ joining $r(\mathbf{p})$ and $r(\mathbf{q})$. Since $L(\pi\circ r\circ\alpha_{a})\leq \mathcal{L}(\pi\circ r\circ\alpha_{a})$  and $\lim_{a\rightarrow 0}\mathcal{L}(\pi\circ r\circ\alpha_{a}) =\bar{d}_{\partial \Omega}(r(\mathbf{p}),r(\mathbf{q}))$, we have that $\lim_{a\rightarrow 0}L(\pi\circ r\circ\alpha_{a})=\bar{d}_{\partial \Omega}(r(\mathbf{p}),r(\mathbf{q}))$. But the length structure $L$ is lower semi continuous, thus $L(\alpha)=\bar{d}_{\partial \Omega}(r(\mathbf{p}),r(\mathbf{q}))$.\fin

\begin{prop}
For every $a>0$, the cosmological level $(S_{a}^{T_{cos}},d_{a}^{T_{cos}})$ is a $CAT(0)$ metric space.
\end{prop}
\begin{proof}
The cosmological hypersurface $S_{a}^{T_{cos}}$ is the graph of a $C^{1}$ convex function $\phi:\mathbb{R}^{n}\rightarrow \mathbb{R}$. Using convolution, one can get a uniform $C^{1}$ approximation of $\phi$ by smooth convex functions $\psi_{i}:\mathbb{R}^{n}\rightarrow \mathbb{R}$. The proposition follows then from \cite[Theorem~II.1A.6]{bridson}. 
\end{proof}

\begin{prop}
The completion $(\Sigma^{\bigstar}, d^{\bigstar}_{\Sigma})$ of the initial singularity $(\Sigma, d_{\Sigma})$ is a $CAT(0)$ metric space.
\end{prop}
 
\begin{proof}
We are first going to proof that $(\Sigma^{\bigstar}, d^{\bigstar}_{\Sigma})$ possesses the approximative midpoints property. For that, it is sufficient to proof it for $(\Sigma,d_{\Sigma})$. 

Let $\mathbf{p}$, $\mathbf{q}$ two points of $X_{T_{cos}}$ the space of gradient lines of the cosmological time $T_{cos}$ and let $\epsilon>0$. For every $a>0$, denote by $p_{a}$ (respectively $q_{a}$) the intersection point of $\mathbf{p}$ and $S_{a}^{T_{cos}}$ (respectively the intersection point of $\mathbf{q}$ and $S_{a}^{T_{cos}}$). Since every $(S_{a}^{T_{cos}},d_{a}^{T_{cos}})$ is geodesic, it possesses the midpoints property. So for every $a>0$, let $z_{a}$ be the point in $S_{a}^{T_{cos}}$ such that $d_{a}(p_{a},z_{a})=d_{a}(q_{a},z_{a})=\frac{1}{2}d_{a}(p_{a},q_{a})$ . For every $a>0$, let us denote by  $\mathbf{z}_{a}$ the cosmological gradient line passing through $z_{a}$. 

By \cite[Proposition~7.6, Proposition~7.8]{bonsante1}, the distances $d_{a}^{T_{cos}}(p_{a},q_{a})$ converge, when $a$ goes to $0$, to $d_{\Sigma}(r(\mathbf{p}),r(\mathbf{q}))$. Then let,
\begin{itemize}
\item $a_{0}>0$ such that for every $0<a\leq a_{0}$ we have $\left|d_{\Sigma}(r(\mathbf{p}),r(\mathbf{q}))-d_{a}(p_{a},q_{a})\right|<\epsilon$;
\item $a_{1}>0$ so that for every $0<a\leq a_{1}$ we have $\left|d_{\Sigma}(r(\mathbf{p}),r(\mathbf{z}_{a_{0}}))-d_{a}(p_{a},\mathbf{z}_{a_{0}})\right|<\frac{\epsilon}{2}$.
\end{itemize}

For every $0<a<\min(a_{0},a_{1})$ we have, $$d_{\Sigma}(r(\mathbf{p}),r(\mathbf{z}_{a_{0}}))\leq d_{a}(p_{a},\mathbf{z}_{a_{0}})+\frac{\epsilon}{2}.$$
But $d_{a}(p_{a},\mathbf{z}_{a_{0}})\leq d_{a_{0}}(p_{a_{0}},z_{a_{0}})$, for $0<a<\min(a_{0},a_{1})$. Hence  $$d_{\Sigma}(r(\mathbf{p}),r(\mathbf{z}_{a_{0}}))\leq \frac{1}{2}d_{a_{0}}(p_{a_{0}},q_{a_{0}})+\frac{\epsilon}{2}\leq \frac{1}{2}d_{\Sigma}(r(\mathbf{p}),r(\mathbf{q}))+\epsilon.$$
In the same way we show that $$d_{\Sigma}(r(\mathbf{q}),r(\mathbf{z}_{a_{0}}))\leq \frac{1}{2}d_{\Sigma}(r(\mathbf{p}),r(\mathbf{q}))+\epsilon.$$ We obtain in this way an $\epsilon$-approximative midpoint $r(\mathbf{z}_{a_{0}})$.

By \cite[Proposition~7.6, Proposition~7.8]{bonsante1}, the $CAT(0)$ metric spaces $(S_{a}^{T_{cos}},d_{a}^{T_{cos}})$ converge in the compact open topology to $(\Sigma,d_{\Sigma})$. Thus the metric spaces $(\Sigma,d_{\Sigma})$ and $(\Sigma^{\bigstar}, d^{\bigstar}_{\Sigma})$ satisfy the $CAT(0)$ $4$-points condition. As $(\Sigma^{\bigstar}, d^{\bigstar}_{\Sigma})$ is complete, by Proposition \ref{prop.cat000000} it is $CAT(0)$.
\end{proof}

\subsection{Asymptotic convergence in the past}

In this part we will prove the last point of  Theorem \ref{mehditheorem2}. Let $M\simeq \Omega/\Gamma$ be a future complete $MGHC$ flat non elementary space-time of dimension $n+1$, where $\Omega$ is a future complete regular domain and $\Gamma$ a discrete subgroup of $SO^{+}(1,n)\ltimes \mathbb{R}^{1,n}$. 

Let $T_{cos}$ be the cosmological time of $\Omega$ and let and $T$ be a quasi-concave $\Gamma$-invariant Cauchy time of $\Omega$. Denote respectively by $X_{T_{cos}}$, $X_{T}$ the space of gradient lines of $T_{cos}$ and the space of gradient lines of $T$. The gradient lines of $T_{cos}$ (respectively $T$) being inextensible temporal curves, they intersect every level set of  $T_{cos}$ (respectively every level set of $T$), which are Cauchy hypersurfaces, at a unique point. It follows that every level set of $T_{cos}$ and every level set of $T$ is identified with the space $X_{T_{cos}}$ and the space $X_{T}$ respectively. Denote by  $d^{T_{cos}}_{a}$ (respectively $\delta^{T_{cos}}_{a}$) the distance of $S^{T_{cos}}_{a}$ transported on $X_{T_{cos}}$ (respectively on $X_{T}$). In the same way we define the distances $d^{T}_{a}$ on $X_{T}$ and $\delta^{T}_{a}$ on $X_{T_{cos}}$. Since the Cauchy hypersurfaces are homeomorphic one to each other, the distances $d^{T_{cos}}_{a}$ and $\delta^{T}_{a}$ (respectively $d^{T}_{a}$ and $\delta^{T_{cos}}_{a}$) define the same topology on $X_{T_{cos}}$ (respectively on $X_{T}$).

The three following results were proved in \cite{mehdi1} (see for instance \cite[Remark~1.2]{mehdi1}).

\begin{prop}
\label{propmehdi77}
The distances $d^{T}_{a}$ defined on $X_{T}$ converge in the compact open topology to a pseudo-distance $d^{T}_{0}$.
\end{prop}

In the case of the cosmological time, the cleaning of the pseudo-metric space $(X_{T_{cos}}, d^{T_{cos}}_{0})$ is isometric to the Initial Singularity $(\Sigma,d_{\Sigma})$. 

\begin{prop}
\label{bel1}
Up to a subsequence, the sequence  $(\delta^{T_{cos}}_{a_{n}})_{n\geq 0}$ (respectively  $(\delta^{T}_{a_{n}})_{n\geq 0}$) converge in the compact open topology to a pseudo-distance $\delta^{T_{cos}}_{0}$ (respectively $\delta^{T}_{0}$) when $n$ goes to $\infty$. Moreover, $$\delta^{T_{cos}}_{0}\leq d^{T}_{0};$$
$$\delta^{T}_{0}\leq d^{T_{cos}}_{0}.$$
\end{prop}

\begin{cor}
\label{bel3}
The marked spectrum of $d^{T_{cos}}_{a}$, $d^{T}_{a}$, $\delta^{T_{cos}}_{0}$ and $\delta^{T}_{0}$ are two by two equals.
\end{cor}

The next proposition gives a more precise description of the behavior of the distances $\delta^{T}_{a}$ near the initial singularity.

\begin{prop}
The distances $\delta^{T}_{a}$, converge in the compact open topology to the pseudo-distance $d^{T_{cos}}_{0}$.
\end{prop}

\begin{proof}
By Proposition \ref{bel1}, it is sufficient to proof that every compact-open  limit point  $\delta^{T}_{0}$ of $(\delta^{T}_{a})_{a>0}$ verifies $\delta^{T}_{0}\geq d^{T_{cos}}_{0}$. 

Let $(\delta^{T}_{a_{i}})_{i\in \mathbb{N}}$  a subsequence of $(\delta^{T}_{a})_{a>0}$ converging to $\delta^{T}_{0}$. Let $\mathbf{p}$ and $\mathbf{q}$ in $X_{T_{cos}}$. For every $i\in \mathbb{N}$, denote respectively by $p_{i}$, $q_{i}$ the intersection points of $S^{T}_{a_{i}}$ and $\mathbf{p}$, $\mathbf{q}$. Note that $J^{-}(p_{i})\cap \overline{\Omega}$  (respectively $J^{-}(q_{i})\cap \overline{\Omega}$ ) is a decreasing sequence of compacts which converge to $r(\mathbf{p})$ (respectively $r(\mathbf{q})$).

Let $i\in \mathbb{N}$, there exists $f(a_{i})$ such that the hypersurface $S^{T_{cos}}_{f(a_{i})}$ is in the past of the hypersurface $S^{T}_{a_{i}}$. Denote respectively by $\mathbf{x_{i}}$, $\mathbf{y_{i}}$ the gradient lines of $T$ passing through the points $p_{i}$, $q_{i}$ of $S^{T}_{a_{i}}$. Let us denote again by $x_{f(a_{i})}$, $y_{f(a_{i})}$ respectively the intersection points of $\mathbf{x_{i}}$ and $\mathbf{y_{i}}$ with $S^{T_{cos}}_{f(a_{i})}$. We get then: $$d^{T_{cos}}_{f(a_{i})}(\mathbf{p},\mathbf{q})\leq \delta^{T_{cos}}_{f(a_{i})}(\mathbf{x_{i}},\mathbf{y_{i}})+d^{T_{cos}}_{f(a_{i})}(\mathbf{p},x_{i})+d^{T_{cos}}_{f(a_{i})}(\mathbf{q},y_{i}).$$
But by Proposition  \ref{propositionmehdi1}, we have, $$\delta^{T_{cos}}_{f(a_{i})}(\mathbf{x_{i}},\mathbf{y_{i}})\leq d^{T}_{a_{i}}(\mathbf{x_{i}},\mathbf{y_{i}})=\delta^{T}_{a_{i}}(\mathbf{p},\mathbf{q})$$ 
Hence $$d^{T_{cos}}_{f(a_{i})}(\mathbf{p},\mathbf{q})\leq \delta^{T}_{a_{i}}(\mathbf{p},\mathbf{q})+d^{T_{cos}}_{f(a_{i})}(\mathbf{p},x_{i})+d^{T_{cos}}_{f(a_{i})}(\mathbf{q},y_{i}).$$
On the one hand we have that $d^{T_{cos}}_{f(a_{i})}(\mathbf{p},x_{i})$ (respectively $d^{T_{cos}}_{f(a_{i})}(\mathbf{q},y_{i})$) is bounded from above by $\left|\left|p_{f(a_{i})}-x_{i}\right|\right|$ (respectively $\left|\left|q_{f(a_{i})}-y_{i}\right|\right|$), where $\left|\left|.\right|\right|$ is the euclidean norm of $\mathbb{R}^{n+1}$. 

But $x_{i}$, $p_{f(a_{i})}$ (respectively $y_{i}$, $q_{f(a_{i})}$) converge when $i$ goes ton $\infty$ to the same point which is $r(\mathbf{p})$ (respectively $r(\mathbf{q})$). This proves that $d^{T_{cos}}_{f(a_{i})}(\mathbf{p},x_{i})$ and $d^{T_{cos}}_{f(a_{i})}(\mathbf{q},y_{i})$ converge to $0$ when $i$ goes toward $\infty$.

On the other hand, the distances $d^{T_{cos}}_{f(a_{i})}$ and $\delta^{T}_{a_{i}}$ converge respectively, when $i$ goes to $\infty$, to $d^{T_{cos}}_{0}$ and $\delta^{T}_{0}$. Thus we have $$d^{T_{cos}}_{0}\leq \delta^{T}_{0}$$ and hence $d^{T_{cos}}_{0}=\delta^{T}_{0}$.
\end{proof}

This proposition proves that the $\Gamma$-metric spaces $(\Gamma,S^{T}_{a},d^{T}_{a})_{a>0}$ converge in the compact open topology, when $a$ goes to $0$, to the initial singularity $(\Gamma,\Sigma,d_{\Sigma})$. Thus the $\Gamma$-metric spaces $(\Gamma,S^{T}_{a},d^{T}_{a})_{a>0}$ converge in the Gromov equivariant topology, when $a$ goes to $0$ to the initial singularity $(\Gamma,\Sigma,d_{\Sigma})$ and hence to its completion $(\Sigma^{\bigstar}, d^{\bigstar}_{\Sigma})$. 

\subsection{Asymptotic convergence in the future}
The object of this part is to proof Theorem \ref{mehditheorem3}. We use the same notation of the previous part. 
Let $T:\Omega\rightarrow\mathbb{R}_{+}$ be a quasi-concave $\Gamma$-invariant Cauchy time. 
\begin{prop}
There exists a constant $C>0$ (depending only on $\Gamma$) such that:
\begin{itemize}
\item 1) for every $C'>C$ , the renormalized distances $\frac{\delta^{T}_{a}}{\sup_{S^{T}_{a}}T_{cos}}$ are, near the infinity, $C'$-quasi-isometric to the hyperbolic metric $d_{\mathbb{H}^{n}}$. In particular, the limit points, for the compact open topology, of the family $(\frac{\delta^{T}_{a}}{\sup_{S^{T}_{a}}T_{cos}})_{a}$ are all $C$-bi-Lipschitz to $d_{\mathbb{H}^{n}}$;
\item 2) In the $2+1$ case, the renormalized $CMC$ distances (respectively $k$ distances) converge for the compact open topology, when times goes to infinity, to the hyperbolic distance $d_{\mathbb{H}^{2}}$.
\end{itemize}
\end{prop}
\begin{proof}
Let $a>0$. Denote by  $a_{+}=\sup_{S^{T}_{a}}T_{cos}$ and by $a_{-}=\inf_{S^{T}_{a}}T_{cos}$. By Proposition \ref{prop.mehdi456677} we have that for every $x$ and $y$ in $X_{T_{cos}}$, $$\frac{a_{-}}{a_{+}}d^{T_{cos}}_{a_{-}}(x,y)\leq\delta^{T}_{a}(x,y)\leq d^{T_{cos}}_{a_{+}}(x,y)$$ So $$(\frac{a_{-}}{a_{+}})^{2}\frac{d^{T_{cos}}_{a_{-}}(x,y)}{a_{-}}\leq\frac{\delta^{T}_{a}(x,y)}{a_{+}}\leq \frac{d^{T_{cos}}_{a_{+}}(x,y)}{a_{+}}$$

1) The general case: by Proposition \ref{propositionmehdi2}, there exists a constant $C'$  such that  $\frac{a_{+}}{a_{-}}\leq C'$ for $a$ big enough. Together with  Proposition \cite[Proposition~7.1]{bonsante1}  we conclude that for $a$ big enough the distance $\delta^{T}_{a}$ is $C'$-quasi-isometric to the hyperbolic hyperbolic $d_{\mathbb{H}^{n}}$. In particular, all the limit points (for the compact open topology) of the family $(\delta^{T}_{a})_{a>0}$ are $C$-bi-Lipschitz to the hyperbolic distance $d_{\mathbb{H}^{n}}$ where $C$ is the constant depending only on $\Gamma$ given in Proposition \ref{propositionmehdi2}.

2) In the $2+1$ case: if $T$ is the $CMC$ time or the $k$-time then by Proposition \ref{propmehdi3} and the Corollary \ref{cormehdi3}, the constant $C$ is equal to one and hence  the family $(\delta^{T}_{a})_{a>0}$ converges in the compact open topology, when $a$ goes to infinity, to $d_{\mathbb{H}^{2}}$.
\end{proof}
This last proposition together with the fact that compact open convergence of $\Gamma$-metric spaces is stronger than the Gromov equivariant convergence conclude the proof of Theorem \ref{mehditheorem3}.

\section{Past convergence in (2+1)-de Sitter space-times}
In this section we will proof Theorem \ref{Theomehdidesianti} in de Sitter case. Let $M\simeq B(S)/\Gamma$ be a $2+1$ dimensional $MGHC$ future complete de Sitter space-time of hyperbolic type, where $B(S)\simeq (\Omega_{1},\mathfrak{g})$ is the associated hyperbolic $dS$-standard spacetime of dimension obtained by a Wick rotation from a flat regular domain $(\Omega,g)$. Let $(\lambda,\mu)$ be the measured geodesic lamination on $\mathbb{H}^{2}$ associated to $(\Omega,g)$. Let's denote respectively by $T_{cos}$ and $\mathcal{T}_{cos}$ the cosmological time of $(\Omega,g)$ and $(\Omega_{1},\mathfrak{g})$.

\begin{prop}
\label{propdesitterbelraouti}
The cosmological level $(\Gamma, S^{\mathcal{T}_{cos}}_{a}, d^{\mathcal{T}_{cos}}_{a})_{a>0}$ converge in the compact open topology, when $a$ goes to $0$, to $(\Gamma,\Sigma,d_{\Sigma})$ the real tree dual to the measured geodesic lamination $(\lambda,\mu)$.
\end{prop}

\begin{proof}
Note that the space of cosmological gradient lines of $(\Omega_{1},\mathfrak{g})$ is the same as the space of cosmological gradient lines of $(\Omega_{1},g)$. Let's denote it by $X_{cos}$. For every $a>0$, the distance $d^{\mathcal{T}_{cos}}_{a}$ of $S^{\mathcal{T}_{cos}}_{a}$ transported to $X_{cos}$ is also denoted by $d^{\mathcal{T}_{cos}}_{a}$.
\\

On the one hand, by Proposition \ref{propmehdi77} the distances $d^{\mathcal{T}_{cos}}_{a}$ and  $d^{T_{cos}}_{\tanh(a)}$ converge respectively in the compact open topology to the pseudo-distances $d^{\mathcal{T}_{cos}}_{0}$ and $d^{T_{cos}}_{0}$ on $X_{cos}$.
\\

On the other hand and for every $a>0$ we have: $d^{\mathcal{T}_{cos}}_{a}(x,y)=\frac{1}{1-\tanh^{2}(a)}d^{T_{cos}}_{\tanh(a)}(x,y)$. Thus, the distances $d^{\mathcal{T}_{cos}}_{a}$ converge  in the compact open topology, when $a$ goes to $0$, to the pseudo-distances $d^{T_{cos}}_{0}$. But the cleaning of $(X_{T_{cos}}, d^{T_{cos}}_{0})$ is isometric to $(\Sigma,d_{\Sigma})$, which is by \cite[Proposition~3.7.2]{benebonsan1} isometric to the real tree  dual to the measured geodesic lamination $(\lambda,\mu)$. So the $\Gamma$ metric spaces $(\Gamma,S^{\mathcal{T}_{cos}}_{a}, d^{\mathcal{T}_{cos}}_{a})_{a>0}$  converge, when $a$ goes to $0$, in the compact open topology to the real tree $(\Gamma,\Sigma,d_{\Sigma})$. Then the $\Gamma$ metric spaces $(\Gamma,S^{\mathcal{T}_{cos}}_{a}, d^{\mathcal{T}_{cos}}_{a})_{a>0}$ converge, when $a$ goes to $0$, in the Gromov equivariant topology to the real tree $(\Gamma,\Sigma,d_{\Sigma})$.

\end{proof}

\textbf{Proof of Theorem \ref{Theomehdidesianti} in the Sitter case}. 
Thanks to Proposition \ref{propositionmehdi1}, Proposition \ref{propdesitterbelraouti} and Remark \ref{jeudivac1}, one can reproduce the proof of Theorem \ref{theomehdi1} without any modification and proves Theorem \ref{Theomehdidesianti} in the de Sitter case.\fin

\section{Past convergence in (2+1)-anti de Sitter space-times}
In this section we will proof Theorem \ref{Theomehdidesianti} in the anti de Sitter case. Let $M\simeq \widetilde{M}/\Gamma$ be the tight past of a $2+1$ dimensional $MGHC$ anti de Sitter space-time, where $\widetilde{M}\simeq (\Omega,\mathfrak{g})$ is obtained by a Wick rotation from a flat regular domain $(\Omega,g)$. Let $(\lambda,\mu)$ be the measured geodesic lamination on $\mathbb{H}^{2}$ associated to $(\Omega,g)$. Let's denote respectively by $T_{cos}$ and $\mathcal{T}_{cos}$ the cosmological time of $(\Omega,g)$ and $(\Omega,\mathfrak{g})$.

\begin{prop}
\label{proplyon}
The cosmological level $(\Gamma, S^{\mathcal{T}_{cos}}_{a}, d^{\mathcal{T}_{cos}}_{a})_{a>0}$ converge in the compact open topology, when $a$ goes to $0$, to $(\Gamma,\Sigma,d_{\Sigma})$ the real tree dual to the measured geodesic lamination $(\lambda,\mu)$.
\end{prop}

\begin{proof}
Note that the space of cosmological gradient lines of $(\Omega_{1},\mathfrak{g})$ is the same as the space of cosmological gradient lines of $(\Omega_{1},g)$. Let's denote it by $X_{cos}$. For every $a>0$, the distance $d^{\mathcal{T}_{cos}}_{a}$ of $S^{\mathcal{T}_{cos}}_{a}$ transported to $X_{cos}$ is also denoted by $d^{\mathcal{T}_{cos}}_{a}$.
\\

On the one hand, by Proposition \ref{propmehdi77}  the distances $d^{\mathcal{T}_{cos}}_{a}$ and  $d^{T_{cos}}_{\tanh(a)}$ converge respectively in the compact open topology to the pseudo-distances $d^{\mathcal{T}_{cos}}_{0}$ and $d^{T_{cos}}_{0}$ on $X_{cos}$.
\\

On the other hand and for every $a>0$ we have: $d^{\mathcal{T}_{cos}}_{a}(x,y)=\frac{1}{1+\tan^{2}(a)}d^{T_{cos}}_{\tanh(a)}(x,y)$. Thus, the distances $d^{\mathcal{T}_{cos}}_{a}$ converge  in the compact open topology, when $a$ goes to $0$, to the pseudo-distances $d^{T_{cos}}_{0}$. So the $\Gamma$ metric spaces $(\Gamma,S^{\mathcal{T}_{cos}}_{a}, d^{\mathcal{T}_{cos}}_{a})_{a>0}$ converge, when $a$ goes to $0$, in the Gromov equivariant topology to the real tree $(\Gamma,\Sigma,d_{\Sigma})$

\end{proof}

\textbf{Proof of Theorem \ref{Theomehdidesianti} in the anti Sitter case}.
Thanks to  Proposition \ref{propositionmehdi1}, Proposition \ref{proplyon} and Remark \ref{jeudivac2}, one can reproduce the proof of Theorem \ref{theomehdi1} without any modification and proves Theorem \ref{Theomehdidesianti} in the anti de Sitter case.\fin

\section{Asymptotic behavior in the Teichmüller space}
The aim object of this part is to proof Theorem \ref{theormehditeich1}. Let $S\simeq \mathbb{H}^{2}/\Gamma$ be a closed hyperbolic surface. Denote by $\operatorname{Teich}(S)$ the Teichmüller space of $S$. On $\operatorname{Teich}(S)$ consider the Teichmüller metric $d_{\operatorname{Teich}}$.
\begin{prop}
\label{proptheich}
Let $g_{1}$ and $g_{2}$ two Riemmannian metric on $S$ such that $(S,g_{1})$ is $K$-bilipchitz to $(S,g_{2})$. Then $d_{\operatorname{Teich}}(\left[g_{1}\right],\left[g_{2}\right])\leq\log K$
\end{prop} 

Let $(\lambda, \mu)$ be a measured geodesic lamination on $S$. Let $M$ be the unique flat (or de Sitter, or the tight past of anti de Sitter) $MGHC$ space-time of dimension $2+1$ associated to $(\lambda, \mu)$. Let $T_{cmc}$ and $T_{k}$ be respectively the $CMC$ time and the $k$ time of $\widetilde{M}$.  For each of the cosmological time, the $k$ time and the $CMC$ time, let us consider respectively the associated curves $a\mapsto \left[g^{T_{cos}}_{a}\right]$, $a\mapsto \left[g^{T_{k}}_{a}\right]$ and $a\mapsto \left[g^{T_{cmc}}_{a}\right]$ in the Teichmüller space $\operatorname{Teich}(S)$ of $S$. 
\begin{prop}\em {\textbf{The flat case}.}
\label{flatlast}
The curves $a\mapsto \left[g^{T_{k}}_{a}\right]$ and $a\mapsto \left[g^{T_{cmc}}_{a}\right]$ converge when $a$ goes to infinity to the hyperbolic
structure $\mathbb{H}^{2}/\Gamma$.
\end{prop}
\begin{proof}
On the one hand and by Proposition \ref{prop.mehdi456677},  $g^{T_{k}}_{a}$ (respectively $g^{T_{cmc}}_{a}$) is $C_{a}^{4}$ bi-Lipschitz to $g^{T_{cos}}_{a}$ for every $a>0$. Moreover by  Proposition \ref{propmehdi3} and the Corollary \ref{cormehdi3}, $C_{a}$ goes to one when $a$ goes to $\infty$. Thus by Proposition \ref{proptheich} we have that  $d_{\operatorname{Teich}}(\left[g^{T_{k}}_{a}\right],\left[g^{T_{cos}}_{a}\right])$ (respectively $d_{\operatorname{Teich}}(\left[g^{T_{cmc}}_{a}\right],\left[g^{T_{cos}}_{a}\right])$) goes to $0$ when $a$ goes to $\infty$.

On the other hand, by a result of Bonsante-Benedetti \cite{benebonsan1}, the cosmological curve $a\mapsto \left[g^{T_{cos}}_{a}\right]$ corresponds to the grafting  associated
to the measured geodesic lamination $(\lambda,\mu)$. The grafting curve converges when times goes to $+\infty$, to the hyperbolic
structure $\mathbb{H}^{2}/\Gamma$. Hence $\left[g^{T_{k}}_{a}\right]$ (respectively  $\left[g^{T_{cmc}}_{a}\right]$) converges when $a$ goes to infinity to to the hyperbolic
structure $\mathbb{H}^{2}/\Gamma$.
\end{proof}



\begin{prop}\em {\textbf{The de Sitter case}.}
The limit points, when time goes to $+\infty$, of the curve $a\mapsto \left[g^{T_{k}}_{a}\right]$ are at bounded Teichmüller distance from  the hyperbolic structure $\mathbb{H}^{2}/\Gamma$.
\end{prop}
\begin{proof}
On the one hand and by Propositions \ref{mehdipropantidesitteralg}, \ref{propbelraoutidesitteralg} we have that $d_{\operatorname{Teich}}(\left[g^{T_{k}}_{a}\right],\left[g^{T_{cos}}_{a}\right])\leq \log(3-H_{0})$ where $H_{0}$ is the constant given in the proof of Proposition \ref{propbelraoutidesitteralg}.

On the other hand $\left[g^{T_{cos}}_{a}\right]$, goes to the grafting metric $\operatorname{gra}_{\lambda}(S)$ when time goes to $+\infty$. Hence the limit points, when $a$ goes to infinity, of $\left[g^{T_{k}}_{a}\right]$ stay at $\log(3-H_{0})$ Teichmüller distance from the grafting metric $\operatorname{gra}_{\lambda}(S)$.
\end{proof}



\nocite{*}

\bibliographystyle{plain}

\bibliography{bibliographiemehdiarticle2}

\end{document}